\documentclass[11pt]{article} 
\usepackage{xcolor}
\usepackage{amsmath}
\usepackage[margin=1in]{geometry}
\usepackage{authblk}
\usepackage{caption}
\usepackage{mathtools}
\usepackage[T1]{fontenc}
\usepackage{graphicx}
\usepackage{amssymb}
\usepackage{amsfonts}
\usepackage{mathrsfs}
\usepackage{comment}
\usepackage{ifthen}
\usepackage{booktabs}
\usepackage{microtype}
\usepackage{array}
\usepackage[caption=false]{subfig}
\usepackage{enumitem}
\usepackage{amsthm}
\usepackage{hyperref}
\hypersetup{%
    linktocpage=true,
    colorlinks=true,
    citecolor=darkgreen,
    linkcolor=blue,
    urlcolor=blue,
}
\usepackage[nameinlink,capitalise]{cleveref}
\newcommand{\email}[1]{\href{mailto:#1}{#1}}
\newenvironment{keywords}{\small \noindent\begin{quote}\textbf{Keywords.}}{\end{quote}}
\newenvironment{MSCcodes}{\vspace{.2cm}\small \noindent \begin{quote}\textbf{MSC codes.}}{\end{quote}}

\theoremstyle{plain}
\newtheorem{theorem}{Theorem}
\newtheorem{lemma}[theorem]{Lemma}
\newtheorem{corollary}[theorem]{Corollary}
\newtheorem{proposition}[theorem]{Proposition}
\newtheorem{remark}[theorem]{Remark}
\newtheorem{assumption}{Assumption}
\crefname{lemma}{Lemma}{Lemmas}
\crefname{remark}{Remark}{Remarks}
\crefname{assumption}{Assumption}{Assumptions}
\crefname{proposition}{Proposition}{Propositions}
\crefname{corollary}{Corollary}{Corollaries}
\crefname{section}{Section}{Sections}
\crefname{subsection}{Subsection}{Subsections}
\Crefname{equation}{Equation}{Equations}
\Crefname{figure}{Figure}{Figures}
\numberwithin{equation}{section}
\numberwithin{theorem}{section}
\numberwithin{table}{section}
\numberwithin{figure}{section}

\DeclarePairedDelimiter\abs{\lvert}{\rvert}
\DeclarePairedDelimiter\ip{\langle}{\rangle}
\DeclarePairedDelimiter\norm{\lVert}{\rVert}
\DeclarePairedDelimiter\paren{\lparen}{\rparen}
\DeclarePairedDelimiter\bkt{\lbrack}{\rbrack}

\DeclareMathOperator{\Id}{Id}
\DeclareMathOperator{\I}{I}

\DeclareMathOperator*{\esssup}{ess\,sup}
\DeclareMathOperator*{\Var}{\mathbf{V}}
\DeclareMathOperator*{\var}{\mathbf{V}}

\newcommand{\placeholder}{\, \cdot \,}

\renewcommand{\d}{d}
\newcommand{\real}{\mathbf R}

\newcommand{\expect}{\mathbf{E}}
\newcommand{\torus}{\mathbf{T}}
\renewcommand{\t}{\mathsf T}

\renewcommand{\leq}{\leqslant}
\renewcommand{\geq}{\geqslant}

\newcommand{\cov}[0]{{\rm cov}}

\newcommand{\ind}{\mathbf{1}}
\renewcommand{\L}{\mathcal{L}}
\newcommand{\X}{\mathcal{X}}
\newcommand{\x}{\mathbf{x}}
\newcommand{\bigO}{\mathrm O}

\newcommand{\e}{\mathrm{e}}
\newcommand{\T}{\torus}
\newcommand{\E}{\expect}
\newcommand{\R}{\real}

\newcommand{\evec}{\mathbf{e}}
\newcommand{\Loss}{\mathscr{L}}
\newcommand{\Z}{\mathbf Z}

\definecolor{darkgreen}{rgb}{0.0, 0.5, 0.0}

\newcommand{\remA}{\mathcal{R}}
\newcommand{\remB}{\mathcal{R}}
\newcommand{\remAT}{\remA_{T,1}}
\newcommand{\remBT}{\remB_{T,2}}
\newcommand{\remATi}{\remA_{T,1}^{(i)}}
\newcommand{\Aterm}{\remAT^{(1)}}
\newcommand{\Bterm}{\remAT^{(2)}}
\newcommand{\Cterm}{\remAT^{(3)}}
\newcommand{\Dterm}{\remBT^{(1)}}
\newcommand{\Eterm}{\remBT^{(2)}}
\newcommand{\Fterm}{\remBT^{(3)}}
\newcommand{\Mterm}{M_T}
\newcommand{\rhoterm}{A_T}

\newcommand{\estTmp}{\widehat{\rho}}
\newcommand{\tmpStatic}{\widehat{\rho}}

\newcommand{\GKest}{\estTmp_{K,T}^{\rm GK}} 
\newcommand{\HEest}{\estTmp_{K,T}^{\rm HE}} 
\newcommand{\HEvar}{\estTmp_T} 
\newcommand{\staticEst}{\tmpStatic_{K,T}^{\rm static}}
\newcommand{\genCorr}[3]{%
  \ifthenelse{\equal{#2}{}}%
    {\widehat{\rho}_{K,T,\psi_{#1}}^{\rm corr,#3}}%
    {\widehat{\rho}_{K,T,\psi_{#1},\psi_{#2}}^{\rm corr,#3}}%
}
\newcommand{\GKcorr}[2]{\genCorr{#1}{#2}{GK}}
\newcommand{\HEcorr}[2]{\genCorr{#1}{#2}{HE}}

\newcommand{\w}{w}

\newcommand{\Lw}{\norm{w'}_{C^0}}
\newcommand{\F}{F^*}

\makeatletter
\def\th@plain{%
  \thm@notefont{}
  \itshape 
}
\def\th@definition{%
  \thm@notefont{}
  \normalfont 
}
\makeatother

\usepackage[url=true,backend=biber,style=trad-abbrv,url=false,doi=false,isbn=false]{biblatex}
\AtEveryBibitem{%
    \clearfield{month}
}
\DeclareFieldFormat{volume}{volume \textbf{#1}}
\DeclareFieldFormat[article]{volume}{\textbf{#1}}

\let\oldparagraph=\paragraph
\renewcommand\paragraph[1]{\oldparagraph{#1.}}

\title{Neural network approaches for variance reduction in fluctuation formulas}
\date{}
\author[1]{G.A.\ Pavliotis$^{a,}$}
\author[2,3]{R.\ Spacek$^{b,}$}
\author[3,2]{G.\ Stoltz$^{c,}$}
\author[2,3]{U.\ Vaes$^{d,}$}
\affil[ ]{\footnotesize $^a$\email{g.pavliotis@imperial.ac.uk},
                        $^b$\email{renato.spacek@enpc.fr},
                        $^c$\email{gabriel.stoltz@enpc.fr},
                        $^d$\email{urbain.vaes@inria.fr}}
\affil[1]{\footnotesize Imperial College London, London, United Kingdom}
\affil[2]{\footnotesize MATHERIALS project-team, Inria Paris}
\affil[3]{\footnotesize CERMICS, \'Ecole des Ponts, IP Paris, France}

\begin{document}
\maketitle

\begin{abstract}
We propose a method utilizing physics-informed neural networks (PINNs) to solve Poisson equations that serve as control variates in the computation of transport coefficients via fluctuation formulas, such as the Green--Kubo and generalized Einstein-like formulas. By leveraging approximate solutions to the Poisson equation constructed through neural networks, our approach significantly reduces the variance of the estimator at hand. We provide an extensive numerical analysis of the estimators and detail a methodology for training neural networks to solve these Poisson equations. The approximate solutions are then incorporated into Monte Carlo simulations as effective control variates, demonstrating the suitability of the method for moderately high-dimensional problems where fully deterministic solutions are computationally infeasible.
\end{abstract}



\section{Introduction}
In this paper,
we investigate general control variate strategies to reduce the variance of estimators of \emph{transport coefficients},
understood in a broad sense as $L^2$ inner products involving the solution to a Poisson equation,
as made precise in~\eqref{eq:tc_ip} below.

\paragraph{Mathematical framework and objective}To precisely describe the mathematical object that the methods we present aim at approximating,
we consider the following general time-homogeneous stochastic differential equation~(SDE) on the state space $\mathcal X$ (where $\X$ is typically $\R^d$ or $\T^d$, with $\T=\R/\Z$ the one-dimensional torus):
\begin{equation}
    dX_t = b(X_t) \, dt + \sigma(X_t) \, dW_t.
    \label{eq:generic_SDE}
\end{equation}
The associated Markov semigroup associated to~\eqref{eq:generic_SDE} is given by
\(
    \paren*{\e^{t \mathcal L} \varphi} (x)
    = \expect \bigl[\varphi(X_t) \big| X_0 = x\bigr].
\)
We denote by~$\mathcal L$ its infinitesimal generator on $L^2(\mu)$,
which is given on $C^{\infty}_{\rm c}(\X)$ by
\begin{equation}
    \notag
    \mathcal L = b^\t\nabla + \frac{1}{2}\sigma\sigma^\t \colon \nabla^2.
\end{equation}
We assume that the dynamics~\eqref{eq:generic_SDE} admits a unique invariant probability measure~$\mu$,
and that the equation~\eqref{eq:generic_SDE} is supplemented with stationarity initial conditions~$X_0\sim \mu$.
Suppose that~$f \colon \mathcal X \to \real$ and $g \colon \mathcal X \to \real$. Our focus in this work is to devise and study computational approaches to approximate quantities of the form
\begin{equation}
    \tag{TC}
    \rho = \ip{f, G},
    \label{eq:tc_ip}
\end{equation}
where $G$ denotes the solution to the Poisson equation (the well-posedness of such equations will be discussed below)
\begin{equation}
    -\L G = g.
    \label{eq:poisson_G_intro}
\end{equation}
For a more convenient presentation of the results to come, we assume throughout this work that the left-hand side of Poisson equations such as \eqref{eq:poisson_G_intro} have average 0 with respect to $\mu$.
This is a necessary condition for the well-posedness of the equation (as the left-hand side integrates to 0 with respect to $\mu$) and can be considered without loss of generality, as one can simply recenter functions with nonzero average by replacing~$g$ by~$g - \E_\mu(g)$. In~\eqref{eq:tc_ip} and throughout this work, scalar products~$\ip{\cdot,\cdot}$ (and associated norms~$\norm{\cdot}$) are taken on~$L^2(\mu)$, unless otherwise specified.

Solving the partial differential equation~\eqref{eq:poisson_G_intro} with deterministic methods is not computationally feasible in high dimensions,
and so a number of probabilistic methods have been proposed in the literature to approximate quantities of the form~\eqref{eq:tc_ip},
some of which we discuss in the next paragraph.
In this work,
we present a variance reduction approach based on control variates for a class of statistical estimators for~$\rho$
relying on the celebrated \emph{Green--Kubo formula}: for two functions~$f,g \in L^2(\mu)$ with average~0 with respect to~$\mu$,
\begin{equation}
    \rho = \int_0^{+\infty} \E\bigl[f(X_0) g(X_t)\bigr] \, dt.
    \label{eq:GK}
\end{equation}
In the rest of this section,
we first present a number of applications where
the calculation of quantities of the form~\eqref{eq:tc_ip} is important.
We then briefly review the control variate approach to variance reduction,
and finally list the contributions of this work.

\paragraph{Applications} Quantities of the form~\eqref{eq:tc_ip} appear in various applications:
\begin{itemize}[leftmargin=*]
    \item
        In molecular dynamics,
        most transport coefficient can be written in this form~\cite{resibois1977}.
        Examples include the mobility, the heat conductivity and the shear viscosity.
        See~\cite[Section 5]{lelievre2016} for an overview,
        and~\cite{roussel2018,spacek2023,blassel2024,pavliotis2023} for recent works proposing efficient computational approaches for the numerical calculation of transport coefficients.
        We emphasize that,
        in molecular dynamics applications,
        it may be rather easy to generate independent and identically distributed (i.i.d.) samples from~$\mu$
        but much more difficult to approximate the transport coefficient~\eqref{eq:tc_ip}.
        A classical example of this is the mobility in the underdamped regime for two-dimensional Langevin dynamics,
        even though the state space is only 4-dimensional;
        see~\cite{pavliotis2008a,pavliotis2023}.

    \item
        In statistics,
        formulas of the form~\eqref{eq:tc_ip} with $f = g$ are important because
        they give the asymptotic variance for the estimator of $\expect_{\mu}[f]$ based on an ergodic average along the realization of the Markov process~\eqref{eq:generic_SDE},
    see~\cite{bhattacharya1982,kipnis1986,cattiaux2012}.
        More precisely, it can be shown that (see~\cite{bhattacharya1982})
        \[
            \frac{1}{\sqrt{T}} \int_{0}^{T} f(X_t) \, \d t
            \xrightarrow[T \to +\infty]{\rm Law} \mathcal N(0, \rho).
        \]
        Thus, the estimation of~$\rho$ is useful for quantifying the uncertainty associated with the ergodic estimator.
        Early works towards this goal are due to Parzen~\cite{parzen1957} and~Neave~\cite{neave1970}.
        See also~\cite{lu2019} for a more modern reference on the subject.

    \item
        In the theory of homogenization for multiscale stochastic differential equations,
        quantities of the form~\eqref{eq:tc_ip} appear in the coefficients of the homogenized equation;
        see~\cite{pardoux2001,pardoux2003},
        as well as~\cite[Section 11]{pavliotis2008} for a theoretical overview,
        and~\cite{vandeneijnden2003,e2005,abdulle2017} for numerical works aimed at the efficient calculation of the coefficients of the homogenized equation.
        In this context, the functions $f$ and $g$ in~\eqref{eq:tc_ip} and~\eqref{eq:poisson_G_intro}
        are usually different,
        which motivates the level of generality we consider.
\end{itemize}

\paragraph{Using control variates for variance reduction} We now motivate and give the flavor of the methods proposed in this work,
without entering into details.
To this end,
assume that an approximation $\psi_g$ to the exact solution~$G$ of~\eqref{eq:poisson_G_intro} is available
and rewrite
\begin{equation}
    \label{eq:decomposition}
    \rho = \ip{f,\psi_g} + \ip{f, G - \psi_g}.
\end{equation}
The first term on the right-hand side is the average of the known function $f \psi_g$ with respect to~$\mu$.
This term can be calculated efficiently,
with classical Monte Carlo methods sampling from~$\mu$.
The second term on the right-hand side of~\eqref{eq:decomposition}, on the other hand,
still involves an inner product between~$f$ and the solution to a Poisson equation of the form~\eqref{eq:poisson_G_intro},
but now with the right-hand side $g + \mathcal L \psi_g$ instead of $g$.
This term can be estimated numerically based on the Green--Kubo formula~\eqref{eq:GK}:
\begin{equation}
    \label{eq:gk_control_variate}
    \ip{f, G - \psi_g} = \int_{0}^{+\infty} \E\bigl[f(X_0) (g + \mathcal L \psi_g) (X_t)\bigr] \, dt.
\end{equation}
A number of statistical estimators can be employed to approximate the integral on the right-hand side,
but we postpone the details of these estimators to~\cref{sec:standard_formulas}.
The point we want to make here is that,
if $\psi_g$ is a good approximation of~$G$ in some sense,
then the function $g + \mathcal L \psi_g$ is~``small'',
in which case it is reasonable to expect estimators for~\eqref{eq:gk_control_variate} to have a smaller variance than those for~\eqref{eq:GK}.

The method just described is one of the three control variate approaches presented in~\cref{sec:cv_methodology}.
These are based on different rewritings of~\eqref{eq:decomposition},
but they all have in common that they require approximate solutions to Poisson equations of the form~\eqref{eq:poisson_G_intro}.

It may seem contradictory,
at this point,
that the control variate approaches we propose are based on approximate solutions to Poisson equation of the form~\eqref{eq:poisson_G_intro},
given our earlier claim that solving~\eqref{eq:poisson_G_intro} deterministically was not computationally feasible in high dimension.
The key to resolve this apparent contradiction is to realize that the control variate approaches described in this work can yield substantial variance reduction even when~$\psi_g$ is a rather  poor approximation of~$G$,
when compared to the usual performance of traditional numerical methods for low-dimensional~PDEs.
Therefore, we argue that using control variates of the type presented in~\cref{sec:cv_methodology} is ideal in moderately high dimensions
-- for which solving the Poisson equation~\eqref{eq:poisson_G_intro} precisely using a deterministic approach is impossible,
but a rough solution can still be calculated,
using for example a neural network.
The applicability of the control variate approach depending on the problem dimension is illustrated in~\cref{table:cv_philosophy}.
The methods presented in~\cref{sec:cv_methodology} can be thought of as hybrid deterministic-probabilistic approaches for calculating~\eqref{eq:tc_ip},
where a rough deterministic approximation is corrected by a Monte Carlo simulation.

\begin{table}[thb]
\footnotesize
    \centering
    \caption{Comparison of fully deterministic, hybrid (probabilistic with control variate) and fully probabilistic methods for estimating~\eqref{eq:tc_ip}.
        The words ``inaccurate'' and ``inefficient'' should be understood as qualifying the performance of the method in comparison with the other methods for the dimension~$d$ considered.
        Note that the values separating moderate dimension from low and high dimensions are of course somewhat arbitrary;
        in particular, the performance of a deterministic method is highly problem-dependent and method-dependent.}
    \label{table:cv_philosophy}
    \begin{tabular}{cccc}
        \toprule
         & \multicolumn{3}{c}{\textbf{Dimension}}  \\
         \cmidrule(lr){2-4}
        \textbf{Method} & Low ($d \leq 3$) & Moderate ($3 < d \leq 10$) & High ($10 < d$) \\ \midrule
        Fully deterministic & \textcolor{darkgreen}{Ideal} & \textcolor{red}{Inaccurate} & \textcolor{red}{Very inaccurate} \\
        Probabilistic + CV & \textcolor{red}{Inefficient} & \textcolor{darkgreen}{Ideal} & \textcolor{red}{Inefficient} \\
        Fully probabilistic & \textcolor{red}{Very inefficient} & \textcolor{red}{Inefficient} & \textcolor{darkgreen}{Ideal}
         \\ \bottomrule
    \end{tabular}
\end{table}

\paragraph{Our contributions}\label{par:Our contributions}
The main contributions of this paper are the following.

\begin{itemize}[leftmargin=*]
    \item
        We review classical estimators for~\eqref{eq:tc_ip} based on the Green--Kubo formula~\eqref{eq:GK},
        and present rigorous results concerning their bias and variance in the longtime limit.
        We focus in the mathematical analysis on estimators based on solutions to the continuous-time dynamics~\eqref{eq:generic_SDE} for simplicity.
        However, the approach we describe can be employed in combination with discrete-time approximations of~\eqref{eq:generic_SDE} used in practice.

    \item
        We present and analyze three control variate approaches to reduce the variance of the classical estimators.
        The first one,
        in the same spirit as in~\cite{roussel2019,pavliotis2023},
        requires an approximate solution to the Poisson equation~\eqref{eq:poisson_G_intro}.
        The second one is a new, computationally more efficient
        approach based on the approximate solution not of~\eqref{eq:poisson_G_intro},
        but of another Poisson equation involving $f$ and the adjoint of~$\mathcal L$.
        Finally, the third approach combines the first two,
        and yields the largest variance reduction.

    \item
        We use a neural network approach to numerically approximate solutions to the Poisson equations required in the control variates.

    \item
        We illustrate the efficiency of the proposed control variates approaches by means of careful numerical experiments.
        Our experiments demonstrate that,
        even when the Poisson equations are not solved accurately,
        substantial variance reduction can still be achieved.
\end{itemize}

\paragraph{Plan of the paper}\label{par:Plan of the paper}
The rest of this paper is organized as follows.
In \cref{sec:standard_formulas},
we present classical estimators for~\eqref{eq:tc_ip} and rigorously establish formulas for their bias and variance.
Then, in~\cref{sec:cv_methodology},
we present novel control variate approaches for these estimators.
Finally, we describe how the approximate solutions to Poisson equation involved in these approaches
can be constructed by using physics-informed neural networks (PINNs), and present numerical experiments in~\cref{sec:numerics}.

\section{Fluctuation formulas}
\label{sec:standard_formulas}
We discuss in this section some fluctuation formulas used for computing \eqref{eq:tc_ip},
then present error bounds for their estimators.
We first discuss the Green--Kubo formula in \cref{subsec:standard_GK},
then a generalized Einstein formula in \cref{subsec:gen_HE}.
A summary of the results presented in this section is given in~\cref{table:summary_fluctuations}.
Additional results on fluctuation formulas are presented in \cref{appendix:add_fluc_form}.

\begin{table}[tb!]
    \footnotesize
    \centering
    \caption{Summary of the results proved in \cref{sec:standard_formulas}.
        For conciseness, the estimators are presented here for just one realization.
    }
    \label{table:summary_fluctuations}
    \begin{tabular}{ccc}
        \toprule
         \textbf{Estimator for} $\ip{f, -\mathcal L^{-1} g}$ & \textbf{Bias} & \textbf{Asymptotic variance}
         \\ \midrule
         \phantom{$\Big($}
             \textbf{Green--Kubo}:
             \cref{subsec:standard_GK}
         &
         \cref{prop:bias_standard_GK}
         &
         \cref{prop:variance_standard_GK}
         \\
         \phantom{$\Bigg($}
         \(
            \displaystyle
             \int_{0}^{T} f(X_0) g(X_t) \, \d t
         \)
         &
         \(
            \displaystyle
            \leq \frac{L}{\lambda} \norm{f} \norm{g} \e^{- \lambda T}
         \)
         &
         \(
            \displaystyle
            \sim 2 \textcolor{red}{T} \norm{f}^2 \ip{g, - \mathcal L^{-1} g}
         \)
         \\
         \midrule
         \phantom{$\Big($}
             \textbf{Half-Einstein}:
         \cref{subsec:gen_HE}
         &
         \cref{prop:bias_gen_standard_HE}
         &
         \cref{prop:var_standard_HE}
         \\
         \phantom{$\Bigg($}
        \(
            \displaystyle
            \int_{0}^{T} \!\!\! \int_{0}^t w\paren*{\frac{t-s}{T}} f(X_s) g(X_t) \, \d s \, \d t
        \)
         &
         \(
         \bigO(T^{-1})
         \)
         &
         \(
         \sim 4 \zeta_w \ip{f, \mathcal -\L^{-1} f}\ip{g, \mathcal -\L^{-1} g}
         \)
         \\
        & &
        \(
            \qquad \text{ with } \zeta_w := \int_{0}^{1} (1-v)\w(v)^2 \, dv
        \)
         \\
         \bottomrule
    \end{tabular}
\end{table}

Before presenting the results,
we introduce useful notation used throughout this section,
as well as an assumption.
Define
\begin{equation}
    \notag
    L^p_0(\mu) = \left\{\varphi \in L^p(\mu) \; \middle| \; \int_\mathcal{X} \varphi \, d\mu = 0 \right\}.
\end{equation}
The notation $\Pi \colon L^2(\mu) \to L^2_0(\mu)$ refers to the orthogonal projection onto $L^2_0(\mu)$, acting as
\[
\Pi f = f - \int_\mathcal{X} f \, d\mu.
\]
We also introduce, for sufficiently smooth $F, G$,
the \emph{carr\'e du champ}
\begin{equation}
    \label{eq:carre_du_champ}
    \Gamma(F, G) =
    \frac{1}{2} \Bigl(\mathcal L(FG) - F \mathcal L G  - G \mathcal L F\Bigr)
    = \frac{1}{2} \nabla F^\t \sigma \sigma^\t \nabla G,
\end{equation}
as well as the associated Dirichlet energy
\[
    \mathcal E(F, G) = \int_{\mathcal X} \Gamma(F, G) \, \d \mu = \frac{1}{2} \Bigl( \ip{F, -\mathcal L G} + \ip{G, -\mathcal L F} \Bigr).
\]
In several proofs in this section,
we will need to manipulate the quantities~$2\Gamma(F, F)$ and~$2 \mathcal E(F, F)$.
Since these occur very often, we introduce the notation~$\chi_F = 2 \Gamma(F, F)$ and~$\overline \chi_F = 2 \mathcal E(F, F)$,
and note that
\begin{equation}
    \label{eq:chibar_ip}
    \overline \chi_F = 2\ip{F,- \mathcal L F}.
\end{equation}
In most of the results of this section,
we make the following assumption, where~$\mathcal{B}(L^2_0(\mu))$ is the Banach space of bounded linear operators on~$L^2_0(\mu)$, and~$\norm{\cdot}_{\mathcal{B}(L^2_0(\mu))}$ the associated operator norm.
\begin{assumption}
    [Decay estimates of the semigroup]
\label{as:semigroup_decay}
    There exist constants~$L\in\R_+$ and~$\lambda>0$ such that
\begin{equation}
        \norm*{\e^{t\L}}_{\mathcal{B}(L^2_0(\mu))} \leq L\e^{-\lambda t}.
        \label{eq:decay_est}
    \end{equation}
\end{assumption}

Decay estimates of the form \eqref{eq:decay_est} can be shown to hold for both overdamped and kinetic Langevin dynamics.
For the overdamped dynamics,
it follows directly from the fact that the probability measure~$\mu$ satisfies a Poincar\'e inequality (in fact, in the setting where~\eqref{eq:generic_SDE} is a reversible diffusion, the decay estimate~\eqref{eq:decay_est} is equivalent to the measure~$\mu$ satisfying a Poincar\'e inequality; see \cite[Chapter 4]{bakry2014} for a comprehensive discussion). For Langevin dynamics,
it can be established from hypocoercivity arguments~\cite{herau2006,dolbeault2009,dolbeault2015,grothaus2016,roussel2018,cao2023,brigati2023,albritton2024}.

\begin{corollary}
    [Well-posedness of Poisson equations]
    \label{cor:well_posedness_poisson}
    Suppose that \cref{as:semigroup_decay} holds. Then~$\L^{-1}$ is a well-defined bounded operator on~$L^2_0(\mu)$,
    which satisfies the following operator identity:
    \begin{equation}
        \notag
        -\L^{-1} = \int_0^{+\infty} \e^{t\L} \, dt.
    \end{equation}
    Thus, Poisson equations of the form $-\L G = g$ for $g \in L^2_0(\mu)$ are well-posed and admit a unique solution in $L^2_0(\mu)$.
\end{corollary}

Throughout this section,
we work under the assumption that the stochastic differential equation~\eqref{eq:generic_SDE}
is supplemented with stationary initial conditions
\begin{equation}
    \notag
    X_0 \sim \mu,
\end{equation}
so that $X_t \sim \mu$ for all~$t \geq 0$.
The results can be generalized beyond stationary initial conditions,
at the cost of an additional term in the bias scaling as $\bigO(T^{-1})$ \cite{mattingly2010}; see also~\cite{joulin2010,eberle2024} for similar results on nonasymptotic relaxation times.
However, to avoid unnecessarily hindering the already nontrivial proofs, we assume that~$X_0\sim \mu$ for the sake of simplicity.

\subsection{Green--Kubo estimator}
\label{subsec:standard_GK}
To numerically estimate the Green--Kubo formula~\eqref{eq:GK}, one typically approximates the expectation as an average over $K$ independent realizations of the reference dynamics~\eqref{eq:generic_SDE} started from i.i.d.\ initial conditions~$X_0\sim\mu$. Additionally, the time integral is truncated at some finite integration time~$T$. This leads to the estimator
\begin{equation}
    \tag{GK}
    \GKest = \frac{1}{K} \sum_{k=1}^K \int_0^T g(X_t^k)f(X_0^k) \, dt.
    \label{eq:GK_estimator}
\end{equation}
We study in this section the properties of this estimator. We first state a result which makes precise the time truncation bias of the estimator \eqref{eq:GK_estimator}.

\begin{proposition}
    [Bounds on bias for the standard Green--Kubo estimator]
\label{prop:bias_standard_GK}
Suppose that \cref{as:semigroup_decay} holds, and that~$f,g \in L^2_0(\mu)$. Then,
\begin{equation}
    \abs*{\E\paren*{\GKest} - \rho} \leq \frac{L}{\lambda}\norm{g}\norm{f}\e^{-\lambda T}.
    \label{eq:statement_bias_std_GK}
\end{equation}
\end{proposition}

This result shows that the bias arising from the time truncation of the integral decays exponentially fast in time. It is a standard result in the literature, see for instance \cite{leimkuhler2016,plechac2022}.
We nonetheless state it, as it will be used to show some error estimates for the improved estimators in \cref{sec:cv_methodology}.

\begin{proof}
    By writing the expectation in terms of the semigroup operator,
    we obtain
\begin{align}
    \notag
    \left|\E\paren*{\GKest} - \rho\right| &= \abs*{\int_0^T \E\bkt*{g(X_t)f(X_0)} \, dt - \int_0^{+\infty} \E\bkt*{g(X_t)f(X_0)} \, dt} \\
    &\leq \int_T^{+\infty} \abs*{\ip{f,\e^{t\L}g}} \, dt.
    \label{eq:bias_std_GK_mid}
\end{align}
By the Cauchy--Schwarz inequality and the decay estimate for the semigroup (\cref{as:semigroup_decay}),
    we have
\begin{align*}
    \abs*{\ip{f,\e^{t\L}g}} \leq \norm{f}\norm*{\e^{t\L}g} \leq L\norm{g}\norm{f}\e^{-\lambda t}.
\end{align*}
Using the above inequality in \eqref{eq:bias_std_GK_mid} yields
\begin{align}
    \notag
    \abs*{\E\paren*{\GKest} - \rho} \leq L\norm{g}\norm{f}\int_T^{+\infty} \e^{-\lambda t} \, dt = \frac{L\norm{g}\norm{f}}{\lambda}\e^{-\lambda T},
\end{align}
which is the desired result \eqref{eq:statement_bias_std_GK}.
\end{proof}

We next make precise the scaling, and exact prefactors, of the asymptotic variance of the estimator~\eqref{eq:GK_estimator}.

\begin{proposition}
    [Bounds on variance for the standard Green--Kubo estimator]
    \label{prop:variance_standard_GK}
    Suppose that \cref{as:semigroup_decay} holds,
    and that $X_0\sim\mu$.
     Consider $f \in L^4_0(\mu)$ and $g \in L^2_0(\mu)$ such that~$G = -\L^{-1}g \in L^4(\mu)$ and~$\chi_G\in L^2(\mu)$. Then,
\begin{equation}
    \notag
    \lim_{T\to+\infty} T^{-1}\Var\paren*{\GKest} = \frac{2\norm{f}^2 \ip{g,-\L^{-1}g}}{K}.
\end{equation}
\end{proposition}

This result shows that the variance scales linearly in time asymptotically as $T \to \infty$.
In regards to the prefactor,
in view of~\cref{cor:well_posedness_poisson} and \cref{as:semigroup_decay} it holds that
\(
    \ip{g,-\L^{-1}g} \leq L\lambda^{-1} \norm{g}^2.
\)
The integrability condition $G \in L^4(\mu)$ can be shown to hold through hypercontractivity, for which a sufficient condition is~$g\in L^4(\mu)$ when decay estimates from \cref{as:semigroup_decay} hold; see \cref{appendix:generalizations_HE_var}, particularly the discussion around \cref{as:semigroup_decay_L4}. Additionally, the condition~$\chi_G\in L^2(\mu)$ holds for instance when the diffusion matrix $\sigma$ is bounded and $\nabla G \in L^4(\mu)$, which can be shown to hold for Langevin dynamics under some conditions on the drift and right-hand side~$g$ (see~\cite{talay2002,kopec2015}).

\begin{proof}
    Since $\Var(\GKest) = K^{-1}\Var(\widehat{\rho}^{\rm GK}_{1,T})$,
    it suffices to prove the result for $K=1$,
    i.e., to prove that
    \begin{equation}
        \label{eq:one_realization}
        \lim_{T \to +\infty} \var \bkt*{\frac{1}{\sqrt{T}} \int_0^T g(X_t)f(X_0) \, \d t}
        = 2 \norm{f}^2 \ip{g, - \L^{-1} g}.
    \end{equation}
    Note first that $G$ is the unique solution to the Poisson equation \eqref{eq:poisson_G_intro}, which is indeed well-posed in view of \cref{cor:well_posedness_poisson}.
    Applying It\^o's formula to~$G(X_t)$, then multiplying the result by $f(X_0)/\sqrt{T}$,
    we obtain
    \begin{align}
        \notag
        \frac{1}{\sqrt{T}} \int_0^T g(X_t)f(X_0) \, dt
        &= \frac{f(X_0)}{\sqrt{T}}\bigl(G(X_0) - G(X_T)\bigr)
        + \frac{f(X_0)}{\sqrt{T}}\int_0^T \nabla G(X_t)^\t \sigma(X_t) \, dW_t \\
        &:= A_T + B_T.
        \label{eq:gk_var_eq1}
    \end{align}
    The first term on the right-hand side of~\eqref{eq:gk_var_eq1} converges to 0 in $L^2(\X)$ as $T\to\infty$ since $f,G \in L^4(\mu)$.
    Applying It\^o isometry to the second term gives
    \begin{align}
        \notag
        \var \bkt*{\frac{f(X_0)}{\sqrt{T}}\int_0^T \nabla G(X_t)^\t \sigma(X_t) \, dW_t}
        &= \frac{1}{T} \int_0^T \E\bkt*{f(X_0)^2\nabla G(X_t)^\t \sigma(X_t)\sigma(X_t)^\t \nabla G(X_t)} \, dt \\
        \notag
        &= \norm{f}^2 \overline \chi_G
        + \frac{1}{T} \int_0^T \E\left[f(X_0)^2 \Pi \chi_G(X_t) \right] \,  dt \\
        \label{eq:gk_var_eq2}
        &= \norm{f}^2 \overline \chi_G
        + \frac{1}{T} \int_0^T \ip*{f^2, \e^{t \mathcal L} \Pi \chi_G} \, dt.
    \end{align}
    The second term in \eqref{eq:gk_var_eq2} tends to 0 as $T\to\infty$ by the decay of the semigroup, since
    \begin{align*}
        \abs*{\frac{1}{T} \int_0^T \ip{f^2,\e^{t\L}\Pi \chi_G} \, dt}
        \leq \frac{C\norm{f}_{L^4(\mu)}^2\norm{\Pi \chi_G}}{T} \int_0^T \e^{-\lambda t} \, dt
        \xrightarrow[T \to +\infty]{} 0.
    \end{align*}
    To conclude,
    we note from the definition of the variance that
    \[
        \sqrt{\var[B_T]} - \sqrt{\var[A_T]}
        \leq \sqrt{\var [A_T + B_T]}
        \leq \sqrt{\var[B_T]} + \sqrt{\var[A_T]}.
    \]
    Taking $T \to +\infty$ leads to the desired result~\eqref{eq:one_realization}
    in view of the expression \eqref{eq:chibar_ip} for~$\chi_G$.
\end{proof}

\subsection{Half-Einstein estimator}
\label{subsec:gen_HE}
In this section, we discuss an Einstein-like estimator of the quantity \eqref{eq:tc_ip}. Einstein formulas allow to rewrite time-integrated correlation functions as
\begin{equation}
    \lim_{T\to\infty} \frac{1}{2T}\E\bkt*{\abs*{\int_0^T f(X_t) \, dt}^2} = \int_0^{+\infty} \E \bigl[ f(X_0)f(X_t) \bigr] \, dt.
    \label{eq:gen_einstein_gk_equiv}
\end{equation}
This equality holds since
\begin{align*}
    \lim_{T\to\infty} \frac{1}{2T}\E\bkt*{\abs*{\int_0^T f(X_t) \, dt}^2}
        &= \lim_{T\to\infty} \frac{1}{2T}\int_0^T \!\!\! \int_0^T\E \bigl[ f(X_s)f(X_t) \bigr] \, ds \, dt \\
    &= \lim_{T\to\infty} \int_0^T\E \bigl[ f(X_0)f(X_t) \bigr] \paren*{1-\frac{t}{T}} \, dt,
\end{align*}
where we used that $\E[f(X_s)f(X_t)] = \E[f(X_0)f(X_{t-s})]$ by stationarity for any $0 \leq s \leq t$,
and where we performed a change of variable in the last step (see \cref{lemma:cov2}).

Evaluating the left-hand side of~\eqref{eq:gen_einstein_gk_equiv} in practice involves two key considerations: (i) approximating the expectation as an average over multiple independent replicas of the system;
and (ii)~considering (moderately) long integration times $T$.

When the quantity of interest is not given by autocorrelations as above, one can consider generalized Einstein formulas of the form
\begin{equation}
    \lim_{T\to\infty} \frac{1}{T}\E \bkt*{\int_0^T \!\!\! \int_0^t f(X_s)g(X_t) \, ds \, dt} = \int_0^{+\infty} \E \bigl[f(X_0)g(X_t)\bigr] \, dt.
    \label{eq:gen_einstein_formula}
\end{equation}
Note that the left-hand side of \eqref{eq:gen_einstein_formula} corresponds to only half of the double integral in the usual Einstein formula, as the quantity of interest is the one-sided correlation~$\E[f(X_0)g(X_t)]$.
Thus, given~$K$ independent replicas of the system with i.i.d.\ initial conditions~$X_0\sim\mu$, a natural estimator of~\eqref{eq:gen_einstein_formula} is given by
\begin{equation}
    \HEest = \frac{1}{TK}\sum_{k=1}^K \int_0^T\!\! \int_0^t f(X_s^k)g(X_t^k) \, ds \, dt.
    \label{eq:HEest_no_w}
\end{equation}
In fact, a more general estimator for~\eqref{eq:gen_einstein_formula} than~\eqref{eq:HEest_no_w} can be formulated in terms of a weight function~$w$.
This gives rise to the following estimator of \eqref{eq:gen_einstein_formula}, realized with $K$ independent trajectories of the reference dynamics with i.i.d.\ initial conditions~$X_0\sim\mu$:
\begin{align}
    \HEest 
    = \frac{1}{TK}\sum_{k=1}^K \int_0^T\int_0^t w\paren*{\frac{t-s}{T}} f(X_s^k)g(X_t^k) \, ds \, dt.
    \label{eq:HE_gen_estimator}
    \tag{HE}
\end{align}
We next state a result which makes precise the bias associated with the estimator~\eqref{eq:HE_gen_estimator}.

\begin{proposition}
    [Bounds on bias for the general half-Einstein estimator]
\label{prop:bias_gen_standard_HE}
Suppose that \cref{as:semigroup_decay} holds true, and that~$X_0\sim \mu$. Consider~$f,g \in L^2_0(\mu)$ and~$w \in C^0[0,1]$. Then,
\begin{equation}
    \abs*{\E\paren*{\HEest} - \rho} \leq L\norm{f}\norm{g}\paren*{T\abs*{\int_0^1 \Bigl( 1 - \w(u)(1 - u) \Bigr)\e^{-\lambda Tu} \, du}+ \frac{\e^{-\lambda T}}{\lambda}}.
    \label{eq:prop_HE_bias_statement}
\end{equation}
\end{proposition}

The result \eqref{eq:prop_HE_bias_statement} suggests that there are two distinct contributions to the bias. The integral term represents the bias due to the weight $\w$, whereas the term $\e^{-\lambda T}$ corresponds to the standard time-truncation bias. While it is possible that the first term vanishes for finite $T$ with an appropriate choice of $\w$, the truncation bias, present on any estimator of \eqref{eq:GK}, only vanishes in the limit~$T\to\infty$.

\begin{remark}
    [Consistency conditions on $w$]
    \label{rem:consistency_w}
    The condition $w(0)=1$ ensures that the estimator~\eqref{eq:HE_gen_estimator} is asymptotically unbiased,
        since the function $u \mapsto \lambda T\e^{-\lambda T u}$ tends to the Dirac measure~$\delta_0$ in the sense of distributions as $T\to\infty$, so that
\begin{align}
    \notag
    T\abs*{\int_0^1 \Bigl(1 - \w(u)(1 - u)\Bigr)\e^{-\lambda Tu} \, du} \xrightarrow[T\to\infty]{} \lambda^{-1}|1 - w(0)|.
\end{align}
This is a standard consistency condition throughout the literature \cite{parzen1957,neave1970,andrews1991}.
\end{remark}

\begin{proof}[Proof of \cref{prop:bias_gen_standard_HE}]
We use the same reasoning as in~\cite[Lemma 2.7]{pavliotis2023}. By stationarity and an appropriate change of variables (see \cref{lemma:cov2}),
it holds that
\begin{align*}
    \E\paren*{\HEest} &= \frac{1}{T}\int_{0}^{T} \int_{0}^t \w\paren*{\frac{t-s}{T}} \E \bigl[ f(X_0) g(X_{t-s}) \bigr] \, \d s \, \d t \\
    &= \frac{1}{T}\int_{0}^{T} \w\paren*{\frac{\theta}{T}} \ip{f,\e^{\theta\L}g} \paren*{T - \theta} \, \d \theta \\
    &= T\int_0^1 w(u)\ip{f,\e^{uT\L}g}(1-u) \, du.
\end{align*}
It follows that
\begin{align}
    \notag
    \abs*{\E\paren*{\HEest} - \rho} &= \abs*{T\int_{0}^1 \w(u) \ip{f,\e^{u T\L}g} \paren*{1 - u} \, \d u - T\int_0^{+\infty}\ip{f,\e^{uT\L}g} \, d u} \\
    &= \abs*{T\int_0^1 \Bigl( 1 - \w(u)(1 - u) \Bigr) \ip{f,\e^{u T\L}g} \, d u + T\int_1^{+\infty} \ip{f,\e^{u T\L}g} \, d u}.
    \label{eq:gen_HE_bias}
\end{align}
By a Cauchy--Schwarz inequality and \cref{as:semigroup_decay},
\begin{align*}
    |\ip{f,\e^{u T\L}}| \leq \|f\|\|g\|\norm*{\e^{u T\L}}_{\mathcal{B}(L^2_0(\mu))} \leq L\|f\|\|g\|\e^{-\lambda T u}.
\end{align*}
Applying this inequality in \eqref{eq:gen_HE_bias} gives
\begin{align*}
    \abs*{\E\paren*{\HEest} - \rho} &\leq L\|f\|\|g\|\paren*{\abs*{T\int_0^1 \Bigl( 1 - \w(u)(1 - u) \Bigr) \e^{-\lambda T u} \, d u} + T\int_1^{+\infty} \e^{-\lambda T u} \, d u} \\
    &= L\|f\|\|g\|\paren*{\abs*{T\int_0^1 \Bigl( 1 - \w(u)(1 - u) \Bigr) \e^{-\lambda T u} \, d u} + \frac{\e^{-\lambda T}}{\lambda}},
\end{align*}
which is the desired result~\eqref{eq:prop_HE_bias_statement}.
\end{proof}

Before stating the next result,
recall that $\chi_F = 2\Gamma(F, F)$,
with $\Gamma$ the \emph{carr\'e du champ} operator defined in~\eqref{eq:carre_du_champ},
and similarly for $\chi_G$.
We now state a result making precise the variance of the generalized half-Einstein estimator~\eqref{eq:HE_gen_estimator}.

\begin{proposition}
    [Bounds on variance for the generalized half-Einstein estimator]
\label{prop:var_standard_HE}
Suppose that \cref{as:semigroup_decay} holds true, and that $X_0\sim\mu$.
Assume that $w \in C^2[0,1]$ and~$w(0)=1$.
Consider~$f,g \in L^4_0(\mu)$ such that $\L^{-1}f, \L^{-1}g \in L^4(\mu)$ and $\chi_F,\chi_G \in L^4(\mu)$,
where $F, G$ denote the solutions in~$L^2_0(\mu)$ to the Poisson equations
\begin{equation}
    \label{eq:poisson}
    - \mathcal L F = f, \qquad - \mathcal L G = g.
\end{equation}
Then, there exists $\mathcal{V}_{f,g}\in\R_+$ such that
\begin{equation}
    \forall T\geq 1, \qquad \Var\paren*{\HEest} \leq \mathcal{V}_{f,g}.
    \label{eq:HE_var_prop_statement_unif}
\end{equation}
Moreover, the asymptotic variance can be made precise as
\begin{equation}
    \lim_{T \to +\infty} \Var\paren*{\HEest} = \frac{4}{K}\paren*{\int_{0}^{1} (1-v)\w(v)^2 \, dv}\ip{f, \mathcal -\L^{-1} f}\ip{g, \mathcal -\L^{-1} g}.
    \label{eq:HE_var_prop_statement_asym}
\end{equation}
\end{proposition}

Note that the right-hand side of \eqref{eq:HE_var_prop_statement_asym} is symmetric in $f,g$.
In contrast with the variance results for Green--Kubo in \cref{prop:variance_standard_GK}, we require stronger integrability conditions on the solution to the Poisson equations, in particular such that~$F,G \in L^4(\mu)$, as well as $\chi_F,\chi_G\in L^4(\mu)$.
These conditions, however, can also be shown to hold with the techniques mentioned in \cref{prop:variance_standard_GK}, namely hypercontractivity arguments discussed in \cref{appendix:generalizations_HE_var} and the framework outlined in \cite{talay2002,kopec2015}.
Additionally, note that the asymptotic variance is constant for the half-Einstein estimator (whereas it scales linearly in time for Green--Kubo), which suggests that long-time integration is not a numerical hindrance in this case.

\begin{proof}The well-posedness of the Poisson equations~\eqref{eq:poisson} is ensured by \cref{cor:well_posedness_poisson}.
It suffices to prove the result for $K=1$, since $\Var(\HEest) = K^{-1}\Var(\widehat{\rho}^{\rm HE}_{1,T})$.
In order to more conveniently state the technical results to come, we write the estimator~\eqref{eq:HE_gen_estimator} for $K=1$ as
\begin{equation}
    \label{eq:estimator}
    \HEvar = \int_0^T I_t  \, g(X_t) \, \d t,
    \qquad I_t := \frac{1}{T}\int_0^t \w\paren*{\frac{t-s}{T}}f(X_s) \, \d s.
\end{equation}
The idea of the proof is to write $\HEvar$ as the sum of a Brownian martingale and perturbation terms,
followed by showing the $L^2(\mu)$ convergence to 0 of each of the perturbation terms.
To this end, let $\xi_t = G(X_t) I_t$. By It\^o's formula, we have
\begin{align*}
    dG(X_t) &= -g(X_t) \, dt + \nabla G(X_t)^\t \sigma(X_t) \, dW_t, \\
    dI_t &= \left( \frac{\w(0)}{T}f(X_t)  + \frac{1}{T^2}\int_0^t \w'\paren*{\frac{t-s}{T}}f(X_s) \, ds \right) \d t.
\end{align*}
Since $d\xi_t = d\bigl(G(X_t)I_t\bigr) = I_tdG(X_t) + G(X_t)dI_t$, it follows that
\begin{equation}
\begin{aligned}
    \d \xi_t &= - g(X_t) I_t \, \d t
    + \frac{1}{T}G(X_t)\paren*{\w(0)f(X_t) + \frac{1}{T}\int_0^t \w'\paren*{\frac{t-s}{T}}f(X_s) \, ds} \, dt \\
    &\qquad + I_t \nabla G(X_t)^\t \sigma(X_t) \, \d W_t.
    \label{eq:dxit}
\end{aligned}
\end{equation}
Integrating \eqref{eq:dxit} between $t=0$ and $t=T$ allows us to write $\HEvar$ as
\begin{equation}
\begin{aligned}
    \HEvar &= - G(X_T) I_T + \frac{1}{T}\int_{0}^T G(X_t)\paren*{\w(0)f(X_t) + \frac{1}{T}\int_0^t \w'\paren*{\frac{t-s}{T}}f(X_s) \, ds} \, dt \\
    &\qquad + \int_{0}^{T} I_t \nabla G(X_t)^\t \sigma(X_t) \, \d W_t.
    \label{eq:HE_est_proof}
\end{aligned}
\end{equation}
At this stage, we would like to replace $I_t$ by some Brownian martingale (and boundary terms).
To this end,
applying It\^o's formula to~$\w((t-s)/T)F(X_s)$,
we obtain
\begin{align*}
    d\biggl[\w\biggl(\frac{t-s}{T}\biggr)F(X_s)\biggr] &= -\frac{1}{T}\w'\paren*{\frac{t-s}{T}}F(X_s) \, ds - \w\paren*{\frac{t-s}{T}}f(X_s) \, ds \\
    &\qquad + w\paren*{\frac{t-s}{T}}\nabla F(X_s)^\t \sigma(X_s) \, dW_s,
\end{align*}
which, after integration in time from $s=0$ to $s=t$, allows us to write $I_t$ as
\begin{equation}
    \begin{aligned}
        I_t &= \frac{1}{T}\left[\w\paren*{\frac{t}{T}}F(X_0) - \w(0)F(X_t)\right] - \frac{1}{T^2}\int_0^t \w'\paren*{\frac{t-s}{T}} F(X_s) \, ds \\
            &\qquad + \frac{1}{T}\int_{0}^{t} \w\paren*{\frac{t-s}{T}} \nabla F(X_s)^\t \sigma(X_s) \, \d W_s.
    \end{aligned}
    \label{eq:It_ito_rewrite}
\end{equation}
Substituting the above equality into \eqref{eq:HE_est_proof} allows us to write $\HEvar$ as the desired sum of a martingale and remainder terms, namely
\begin{equation}
    \notag
    \HEvar = \frac{\rhoterm}{T} + \frac{\Mterm}{T} + \frac{\remAT}{T} + \frac{\remBT}{T^2},
\end{equation}
where $\rhoterm/T$ denotes a term which will converge to $\rho$:
\begin{equation*}
    \rhoterm = \w(0)\int_{0}^T G(X_t) f(X_t) \, \d t,
\end{equation*}
the martingale $\Mterm/T$ is the dominant term among the other terms:
\begin{equation}
    \notag
    \Mterm = \int_{0}^{T} \paren*{\int_{0}^{t} \w\paren*{\frac{t-s}{T}}\nabla F(X_s)^\t \sigma(X_s) \, dW_s} \nabla G(X_t)^\t \sigma(X_t) \, dW_t,
\end{equation}
and $\remAT/T$ and $\remBT/T^2$ are remainder terms that
we group according to the power of~$1/T$ that they multiply:
\begin{align*}
    \begin{split}
    \remAT &= G(X_T)\bigl(\w(0)F(X_T) - \w(1)F(X_0)\bigr) \\
    &\qquad - G(X_T)\int_{0}^{T} \w\paren*{\frac{T-s}{T}}\nabla F(X_s)^\t \sigma(X_s) \, dW_s \\
    &\qquad - \int_{0}^T  \bkt*{\w(0)F(X_t) - \w\paren*{\frac{t}{T}}F(X_0)} \nabla G(X_t)^\t \sigma(X_t) \, dW_t
    \end{split} \\
    &:= \Aterm + \Bterm + \Cterm,
\end{align*}
\begin{align*}
    \begin{split}
    \remBT &= G(X_T)\int_0^T \w'\paren*{\frac{T-s}{T}} F(X_s) \, ds \\
    &\qquad + \int_0^T G(X_t) \biggl(\int_0^t \w'\paren*{\frac{t-s}{T}}f(X_s) \, ds \biggr) \, dt \\
    &\qquad - \int_0^T \biggl(\int_0^t \w'\paren*{\frac{t-s}{T}} F(X_s) \, ds\biggr) \nabla G(X_t)^\t \sigma(X_t) \, dW_t
    \end{split} \\
    &:= \Dterm + \Eterm + \Fterm.
\end{align*}
By the triangle inequality,
\begin{equation}
\label{eq:var_sqr_split}
\begin{split}
    &\abs*{\E\bkt*{\Bigl\lvert \HEvar - \E[\HEvar] \Bigr\rvert^2}^{\frac{1}{2}} - \E\bkt*{\left\lvert \frac{M_T}{T} \right\rvert^2}^{\frac{1}{2}}} \\
    &\qquad \leq \E\bkt*{\left\lvert\frac{\remAT}{T} \right\rvert^2}^{\frac{1}{2}} + \E\bkt*{\left\lvert\frac{\remBT}{T^2} \right\rvert^2}^{\frac{1}{2}} + \E\bkt*{\abs*{\frac{A_T}{T} - \E[\HEvar]}^2}^{\frac{1}{2}}.
\end{split}
\end{equation}
Note that if the right-hand side of \eqref{eq:var_sqr_split} vanishes as $T\to\infty$,
and if the variance of~$M_T/T$ admits a limit as $T \to +\infty$,
then
\begin{equation}
    \notag
    \lim_{T\to\infty}\E\Bigl[\bigl\lvert \HEvar - \E[\HEvar] \bigr\rvert^2\Bigr] = \lim_{T\to\infty}\E\left[\left|\frac{\Mterm}{T}\right|^2\right].
\end{equation}
The claimed result~\eqref{eq:HE_var_prop_statement_asym} follows by showing the $L^2$ convergence of each term in~\eqref{eq:var_sqr_split}, namely
\begin{subequations}
\begin{align}
    \notag
    \lim_{T\to\infty}\E\left[\abs*{\frac{\rhoterm}{T} - \E[\HEvar]}^2\right] &= 0, \\
    \label{eq:three_limits_M}
    \lim_{T\to\infty}\E\left[\left|\frac{\Mterm}{T}\right|^2\right] &= 4\paren*{\int_{0}^{1} (1-v)\w(v)^2 \, dv} \ip{f, -\L^{-1} f}\ip{g, -\L^{-1} g}, \\
    \lim_{T\to\infty}\E\left[\left|\frac{\remAT}{T}\right|^2\right] &= \lim_{T\to\infty}\E\left[\left|\frac{\remBT}{T^2}\right|^2\right] = 0.
    \label{eq:three_limits_AB}
\end{align}
\end{subequations}
We successively prove the various limits in the order they appear above.

\paragraph{Convergence of $\rhoterm/T$}
Note first that $\expect \left[ f(X_t) G(X_t) \right] = \rho$.
Since $w(0) = 1$,
it holds that
\begin{align*}
    \expect \bkt*{\left\lvert \frac{A_T}{T} - \rho \right\rvert^2}
    &= \frac{1}{T^2} \int_{0}^T \!\!\! \int_{0}^{T} \expect \left[ \Bigl(f(X_s) G(X_s) - \rho\Bigr) \Bigl(f(X_t) G(X_t) - \rho\Bigr) \right]\, \d s \, \d t \\
    &= \frac{2}{T^2} \int_{0}^T (T-\theta) \expect \left[ \Bigl(f(X_0) G(X_0) - \rho\Bigr) \Bigl(f(X_{\theta}) G(X_{\theta}) - \rho\Bigr) \right]\, \d \theta.
\end{align*}
Using the decay of the semigroup in~\cref{as:semigroup_decay}, and noting that $fG - \rho = \Pi(fG) \in L^2_0(\mu)$,
we deduce that
\begin{align*}
    \expect \bkt*{\left\lvert \frac{A_T}{T} - \rho \right\rvert^2}
    &= \frac{2}{T^2} \int_{0}^T (T-\theta) \Bigl\langle f G- \rho, \e^{\theta \mathcal L} (fG - \rho) \Bigr\rangle \, \d \theta \\
    & \leq \frac{2\norm{fG - \rho}^2}{T^2} \int_{0}^T L \e^{- \lambda \theta} \, \d \theta
    \xrightarrow[T \to +\infty]{} 0.
\end{align*}

\paragraph{Convergence of the dominant term $\Mterm/T$} We next show the convergence \eqref{eq:three_limits_M} of the dominant term $\Mterm/T$. To this end, we introduce
\begin{equation*}
    N_S :=  \int_{0}^{S} \w\paren*{\frac{S - s}{T}} \nabla F(X_s)^\t \sigma(X_s) \, \d W_s.
\end{equation*}
Recall the definition of $\chi_G$
By It\^o isometry, we have that
\begin{align}
    \notag
    \E\bkt*{\abs*{\frac{\Mterm}{T}}^2}
    &= \frac{1}{T^2} \int_{0}^{T} \E\bkt*{N_t^2 \chi_G(X_t)} \, dt \\
    &= \frac{\overline{\chi}_G}{T^2}\int_{0}^{T} \E\bkt*{N_{t}^2} \, dt + \frac{1}{T^2}\int_{0}^{T} \E\bkt*{N_{t}^2 \Pi \chi_G(X_{t})} \, dt.
    \label{eq:M_T_chi_split}
\end{align}
We also recall that $\overline{\chi}_G = 2\ip{g, \mathcal -\L^{-1} g}$ from \eqref{eq:chibar_ip},
and note that $\E[\chi_G(X_t)] = \overline{\chi}_G$ since by assumption~$X_0\sim \mu$.
The first term in~\eqref{eq:M_T_chi_split} is then equal to the right-hand side of~\eqref{eq:HE_var_prop_statement_asym} since,
by It\^o isometry, stationarity and an appropriate change of variable (see \cref{cor:cov3} for details),
\begin{align*}
\frac{\overline{\chi}_G}{T^2}\int_{0}^{T} \E\bkt*{N_{t}^2} \, dt &= \frac{\overline{\chi}_G}{T^2}\int_0^T\int_0^t \E[\chi_F] \w\paren*{\frac{t-s}{T}}^2 \, ds \, dt \\
    &= 4 \ip{f, -\L^{-1} f}\ip{g, -\L^{-1} g}\frac{1}{T^2}\int_{0}^T \int_0^{t} \w\paren*{\frac{t-s}{T}}^2 \, ds \, \d t \\
    &= 4 \ip{f, -\L^{-1} f}\ip{g, -\L^{-1} g}\int_{0}^1 (1-v)\w(v)^2 \, dv.
\end{align*}
It remains to show that the second term in \eqref{eq:M_T_chi_split} converges to 0 as $T \to +\infty$. Letting~$z=t/T$, this term can be rewritten as
\begin{equation*}
    \frac{1}{T}\int_{0}^{1} \E\bkt*{N_{zT}^2 \Pi \chi_G(X_{zT})} \, dz = \int_{0}^{1} z\E\bkt*{\frac{N_{zT}^2}{zT} \Pi \chi_G(X_{zT})} \, dz.
\end{equation*}
In order to show that the integral above vanishes, we first show that its integrand vanishes, i.e.,
\begin{equation}
    \lim_{S \to +\infty} \E\bkt*{\frac{N_S^2}{S} \Pi \chi_G(X_{S})} = 0.
    \label{eq:dom_conv_first_lim}
\end{equation}
To this end, to make use of the exponential decay of the semigroup, we rewrite
\begin{align}
    \frac{N_S^2}{S} \Pi \chi_G(X_{S}) &= \frac{N_{S-\sqrt{S}}^2}{S}  \Pi \chi_G(X_S)
    + \left(\frac{N_S^2}{S} - \frac{N_{S-\sqrt{S}}^2}{S}\right) \Pi \chi_G(X_S).
    \label{eq:cond_splitting}
\end{align}
For the first term, by the Cauchy--Schwarz and Burkholder--Davis--Gundy (BDG) inequalities (with BDG constant $C_4$),
we have
\begin{align*}
    \E\bkt*{\frac{N_{S-\sqrt{S}}^2}{S} \Pi \chi_G(X_S)}
    &= \E\bkt*{\frac{N_{S-\sqrt{S}}^2}{S} \paren*{\e^{\sqrt{S} \mathcal L} \Pi \chi_G}(X_{S-\sqrt{S}})} \\
    &\leq \E\bkt*{\frac{N_{S-\sqrt{S}}^4}{S^2}}^{\frac{1}{2}} \E\bkt*{\abs*{\e^{\sqrt{S} \mathcal L} \Pi \chi_G(X_{S-\sqrt{S}})}^2}^{\frac{1}{2}} \\
    &\leq \sqrt{C_4}\E\bkt*{\frac{\ip*{N}_{S-\sqrt{S}}^2}{S^2}}^{\frac{1}{2}} \E\bkt*{\abs*{\e^{\sqrt{S} \mathcal L} \Pi \chi_G(X_{S-\sqrt{S}})}^2}^{\frac{1}{2}},
\end{align*}
which vanishes as $S\to\infty$, as the first term is uniformly bounded for $S\geq 1$ since~$\chi_F\in L^2(\mu)$:
\begin{align*}
    \E\bkt*{\frac{\ip*{N}_{S-\sqrt{S}}^2}{S^2}} &= \frac{1}{S^2}\E\bkt*{\abs*{\int_0^{S-\sqrt{S}} w\paren*{\frac{S-\sqrt{S}-s}{T}}^2 \chi_F(X_s) \, ds}^2} \\
    &\leq \frac{1}{S}\E\bkt*{\int_0^{S-\sqrt{S}} w\paren*{\frac{S-\sqrt{S}-s}{T}}^4 \chi_F(X_s)^2 \, ds} \\
     &\leq \frac{\norm{w}_{C^0}^4}{S}\int_0^{S-\sqrt{S}} \E\bkt*{\chi_F(X_s)^2} \, ds \leq \norm{w}_{C^0}^4 \norm{\chi_F}^2;
\end{align*}
while the second term tends to 0 by the decay of the semigroup and stationarity
since~$\chi_G\in L^2(\mu)$:
\begin{align*}
    \E\bkt*{\abs*{\e^{\sqrt{S} \mathcal L} \Pi \chi_G(X_{S-\sqrt{S}})}^2}^{\frac{1}{2}} = \norm*{\e^{\sqrt{S} \mathcal L} \Pi \chi_G}
    \leq L\norm{\Pi \chi_G}\e^{-\lambda \sqrt{S}} \xrightarrow[S\to\infty]{} 0.
\end{align*}
For the second term in \eqref{eq:cond_splitting}, we have
\begin{align}
    \notag
    &\E\bkt*{\left(\frac{N_S^2}{S} - \frac{N_{S-\sqrt{S}}^2}{S}\right) \Pi \chi_G(X_S)}
    = \E\bkt*{\frac{N_S - N_{S - \sqrt{S}}}{\sqrt{S}} \frac{N_S + N_{S-\sqrt{S}}}{\sqrt{S}} \Pi \chi_G(X_S)} \\
    \label{eq:split_NS_2}
    &\qquad \qquad \qquad \leq \E\bkt*{\frac{\abs*{N_S - N_{S - \sqrt{S}}}^2}{S}}^{\frac{1}{2}}
    \E\bkt*{\frac{\abs*{(N_S + N_{S - \sqrt{S}})\Pi \chi_G(X_S)}^2}{S}}^{\frac{1}{2}}.
\end{align}
To show that \eqref{eq:split_NS_2} vanishes, we show that the first term vanishes, while the second one is bounded uniformly in $S$.
We start by considering the first term:
\begin{equation}
\begin{aligned}
    N_S - N_{S - \sqrt{S}} &= \int_{S-\sqrt{S}}^S \w\paren*{\frac{S-s}{T}}\nabla F(X_s)^\t \sigma(X_s) \, dW_s \\
    &\qquad + \int_0^{S-\sqrt{S}} \bkt*{\w\paren*{\frac{S-s}{T}} - \w\paren*{\frac{S-\sqrt{S}-s}{T}}}\nabla F(X_s)^\t \sigma(X_s) \, dW_s.\label{eq:NS_split_N2}
\end{aligned}
\end{equation}
The first right-hand side term in \eqref{eq:NS_split_N2} is $\bigO(\sqrt{S})$ in~$L^2$ by It\^o isometry and the fact that $\w$ is uniformly bounded, while the second term is $\bigO(1)$ since $\w$ is Lipschitz (with Lipschitz constant~$\Lw$):
\begin{align*}
    \E\bkt*{|N_S - N_{S-\sqrt{S}}|^2} &\leq 2\E\bkt*{\int_{S-\sqrt{S}}^S \w\paren*{\frac{S-s}{T}}^2 \chi_F(X_s) \, ds}
    + \frac{2S\Lw^2}{T^2}\E\bkt*{\int_0^{S-\sqrt{S}} \chi_F(X_s) \, ds} \\
    &\leq 2\sqrt{S}\norm{w}^2_{C^0}\overline{\chi}_F + 2\Lw^2\overline{\chi}_F.
\end{align*}
Therefore, the first term in \eqref{eq:split_NS_2} is of order $S^{-\frac{1}{4}}$ (and uniformly bounded for $S\geq 1$).
We now show that the second term in \eqref{eq:split_NS_2} is uniformly bounded.
Indeed, by applying the Cauchy--Schwarz and BDG inequalities, and since~$\chi_F,\chi_G\in L^4(\mu)$ and $w$ is uniformly bounded,
\begin{align*}
&\E\bkt*{\left\lvert (N_S + N_{S - \sqrt{S}})\Pi \chi_G(X_S) \right\rvert^2} \leq 2\E\bkt*{N_S^2 \Pi \chi_G(X_S)^2} + 2\E\bkt*{N_{S - \sqrt{S}}^2\Pi \chi_G(X_S)^2} \\
    &\leq 2\sqrt{C_4}\E\bkt*{\ip*{N}_S^2}^{\frac{1}{2}} \E\bkt*{\Pi  \chi_G(X_s)^4}^{\frac{1}{2}} + 2\sqrt{C_4}\E\bkt*{\ip*{N}_{S-\sqrt{S}}^2}^{\frac{1}{2}} \E\bkt*{\Pi \chi_G(X_S)^4}^{\frac{1}{2}}  \\
    &\leq 4\sqrt{C_4} S\norm{w}^2_{C^0}\norm{\chi_F}\norm{\Pi  \chi_G}_{L^4(\mu)}^2,
\end{align*}
so we deduce that~\eqref{eq:split_NS_2} is uniformly bounded in $S$ and tends to 0 as~$S \to +\infty$, and thus~\eqref{eq:dom_conv_first_lim} holds. The functions appearing in the integral are uniformly bounded, hence uniformly integrable. Using dominated convergence, we conclude that
\begin{equation}
    \notag
    \int_{0}^1 z \E\bkt*{\frac{N_{zT}^2}{zT} \Pi \chi_G(X_{zT})} \, dz \xrightarrow[T \to +\infty]{} 0.
\end{equation}
This proves \eqref{eq:three_limits_M}.

\paragraph{Convergence of remainder terms $\remAT/T$ and $\remBT/T^2$} In order to conclude the desired result~\eqref{eq:HE_var_prop_statement_asym}, it remains to prove \eqref{eq:three_limits_AB}, i.e., that the remainder terms $\remAT/T$ and $\remBT/T^2$ vanish in squared expectation as~$T\to +\infty$.
Let us start with $\remAT$.
We show that for each $i \in \{1, 2, 3\}$,
the quantity $\E[|\remATi|^2]/T$ is uniformly bounded in~$T$:
\begin{itemize}[leftmargin=*]
    \item
            $\Aterm$:
            Since $G,F \in L^4(\mu)$ and $w$ is  uniformly bounded, it trivially holds that~
            \[
                \E\left[|\Aterm|^2\right] = \bigO(1).
            \]

    \item
            $\Bterm$:
Applying the Cauchy--Schwarz and BDG inequalities (and since~$G\in L^4(\mu), \chi_F\in L^2(\mu)$ and~$w$ is uniformly bounded), it follows that
\begin{align*}
        \E\left[ \left\lvert \Bterm \right\rvert ^2\right] &\leq \E\left[G(X_T)^4\right]^\frac{1}{2}\E\left[\abs*{\int_0^T \w\paren*{\frac{T-s}{T}}\nabla F(X_s)^\t \sigma(X_s) \, dW_s}^4\right]^\frac{1}{2} \\
        &\leq \sqrt{C_4}\norm{G}_{L^4(\mu)}^2\E\left[\abs*{\int_0^T \w\paren*{\frac{T-s}{T}}^2\chi_F(X_s) \, ds}^2\right]^\frac{1}{2} \\
        &\leq \sqrt{C_4}\sqrt{T}\norm{G}_{L^4(\mu)}^2 \E\left[\int_0^T \w\paren*{\frac{T-s}{T}}^4\chi_F(X_s)^2 \, ds\right]^\frac{1}{2} \\
        &\leq \sqrt{C_4}T\norm{w}^2_{C^0} \norm{G}_{L^4(\mu)}^2\norm{\chi_F}.
    \end{align*}
\item $\Cterm$: By the It\^o isometry and a Cauchy--Schwarz inequality, since $F\in L^4(\mu), \chi_G\in L^2(\mu)$ and~$w$ is uniformly bounded, it follows that
\begin{align*}
            \E\bkt*{\abs*{\Cterm}^2}
            &\leq 2\int_{0}^T \E\bkt*{F(X_t)^2\chi_G(X_t)} \, dt + 2\int_0^T \E\bkt*{\w\paren*{\frac{t}{T}}^{2} F(X_0)^2 \chi_G(X_t)} \, dt \\
            &\leq 2T\norm{F}_{L^4(\mu)}^2\norm{\chi_G} + 2T\norm{w}_{C^0}^2 \norm{F}_{L^4(\mu)}^2\norm{\chi_G}.
    \end{align*}
    \end{itemize}
Thus, we conclude that $\E[|\remAT|^2] = \bigO(T)$ as desired.
        Let us now consider $\remBT$.
        We have
\begin{itemize}[leftmargin=*]
    \item $\Dterm$: By Cauchy--Schwarz and the fact that $w'$ is uniformly bounded and $F,G\in L^4(\mu)$, it holds that
\begin{align*}
            \E\bkt*{\abs*{\Dterm}^2}
            &= \E\left[G(X_T)^2\left(\int_0^T w'\paren*{\frac{T-s}{T}}F(X_s) \, ds\right)^2\right] \\
            &\leq \E\left[G(X_T)^4\right]^\frac{1}{2}\E\bkt*{\abs*{\int_0^T w'\paren*{\frac{T-s}{T}}F(X_s) \, ds}^4}^\frac{1}{2} \\
            &\leq T^2\norm{w'}^2_{C^0}\norm{G}_{L^4(\mu)}^2\norm{F}_{L^4(\mu)}^2.
        \end{align*}
\item $\Eterm$: By the same reasoning we used to write \eqref{eq:It_ito_rewrite},
        we apply It\^o's formula to~$\w'((t-s)/T)F(X_s)$ to rewrite~$\Eterm$ as
\begin{equation}
    \begin{aligned}
        \Eterm &= \int_0^T G(X_t)\left(\w'\paren*{\frac{t}{T}}F(X_0) - \w'(0)F(X_t)\right) dt \\
        &\qquad - \frac{1}{T}\int_0^T G(X_t)\paren*{\int_0^t \w''\paren*{\frac{t-s}{T}}F(X_s) \, ds} \, dt \\
        &\qquad + \int_0^T G(X_t)\left(\int_0^t \w'\paren*{\frac{t-s}{T}}\nabla F(X_s)^\t \sigma(X_s) \, dW_s\right) dt.
        \label{eq:Eterm_eq1}
    \end{aligned}
    \end{equation}
The expectation of the square of the first two terms in \eqref{eq:Eterm_eq1} are easily shown to be $\bigO(T^2)$ by a Cauchy--Schwarz inequality, since $\w',\w''$ are uniformly bounded and~$F,G\in L^4(\mu)$:
\begin{align*}
        &\E\left[\abs*{\int_0^T G(X_t)\paren*{\w'\paren*{\frac{t}{T}}F(X_0) - \w'(0)F(X_t)} \, dt}^2\right] \\
        &\qquad \leq 2T\norm{w'}_{C^0}^2\paren*{\int_0^T \E\bkt*{G(X_t)^2F(X_0)^2} \, dt + \int_0^T \E\bkt*{G(X_t)^2F(X_t)^2} \, dt} \\
&\qquad \leq 4T^2\norm{w'}_{C^0}^2 \norm{G}_{L^4(\mu)}^2\norm{F}_{L^4(\mu)}^2,
    \end{align*}
and
\begin{align*}
        &\frac{1}{T^2}\E\left[\abs*{\int_0^T G(X_t)\paren*{\int_0^t \w''\paren*{\frac{t-s}{T}}F(X_s) \, ds} dt}^2\right] \\
        &\qquad \leq \frac{1}{T}\int_0^T \E\bkt*{G(X_t)^4}^\frac{1}{2}\E\bkt*{\paren*{\int_0^t\w''\paren*{\frac{t-s}{T}}F(X_s) \, ds}^4}^\frac{1}{2} \, dt \\
        &\qquad \leq \frac{\norm{w''}_{C^0}^2\norm{F}_{L^4(\mu)}^2}{T}\int_0^T t^2 \E\left[G(X_t)^4\right]^\frac{1}{2} dt
        = \frac{T^2\norm{w''}_{C^0}^2\norm{F}_{L^4(\mu)}^2\norm{G}_{L^4(\mu)}^2}{3}.
    \end{align*}
For the third term, since $\chi_F\in L^4(\mu)$, applying Cauchy--Schwarz then BDG gives
\begin{align*}
        &\E\left[\left\lvert\int_0^T G(X_t)\paren*{\int_0^t \w'\paren*{\frac{t-s}{T}}\nabla F(X_s)^\t \sigma(X_s) \, dW_s} \, dt\right\rvert^2\right] \\
        &\qquad \leq T\int_0^T \E\bkt*{G(X_t)^4}^\frac{1}{2}\E\bkt*{\abs*{\int_0^t \w'\paren*{\frac{t-s}{T}}\nabla F(X_s)^\t \sigma(X_s) \, dW_s}^4}^\frac{1}{2} dt \\
        &\qquad \leq T\sqrt{C_4}\int_0^T \E\bkt*{G(X_t)^4}^\frac{1}{2}\E\bkt*{\abs*{\int_0^t \w'\paren*{\frac{t-s}{T}}^2\chi_F(X_s) \, ds}^2}^\frac{1}{2} dt \\
        &\qquad \leq T\sqrt{C_4}\norm{w'}_{C^0}^2\norm{\chi_F} \int_0^T t\E\bkt*{G(X_t)^4}^\frac{1}{2} dt = \frac{T^3\sqrt{C_4}}{2}\norm{w'}_{C^0}^2\norm{\chi_F}\norm{G}_{L^4(\mu)}^2,
    \end{align*}
so it holds that $\E[|\Eterm|^2] = \bigO(T^3)$.

    \item $\Fterm$: By It\^o isometry, a Cauchy--Schwarz inequality, the fact that $\w'$ is uniformly bounded, a H\"older inequality and $F \in L^4(\mu), \chi_G\in L^2(\mu)$, it follows that
\begin{align*}
        \E\bkt*{\abs*{\Fterm}^2} &= \E\bkt*{\int_0^T \paren*{\int_0^t \w'\paren*{\frac{t-s}{T}} F(X_s) \, ds}^2 \chi_G(X_t) \, dt} \\
&\qquad \leq \norm{\chi_G} \int_0^T \E\bkt*{\abs*{\int_0^t \w'\paren*{\frac{t-s}{T}} F(X_s) \, ds}^4}^\frac{1}{2} dt \\
        &\qquad \leq \norm{\chi_G}\norm{w'}_{C^0}^2 \int_0^T \E\bkt*{ \left(\int_0^t \abs*{F(X_s)} \, ds\right)^4}^\frac{1}{2} dt \\
        &\qquad \leq \norm{\chi_G}\norm{w'}_{C^0}^2 \int_0^T \E\bkt*{t^3\int_0^t |F(X_s)|^4 \, ds}^\frac{1}{2} dt \\
        &\qquad \leq \norm{\chi_G}\norm{w'}_{C^0}^2\norm{F}_{L^4(\mu)}^2 \int_0^T t^2 \, dt
        = \frac{T^3}{3}\norm{\chi_G}\norm{w'}_{C^0}^2\norm{F}_{L^4(\mu)}^2.
    \end{align*}
\end{itemize}
Thus, we deduce that $\E[|\remBT|^2]/T^4$ tends to 0 as $T\to\infty$ (which allows to conclude that the right-hand side of \eqref{eq:var_sqr_split} converges to 0),
and in particular that it is uniformly bounded in $T$, which gives \eqref{eq:HE_var_prop_statement_unif}.
This concludes the proof.
\end{proof}

\begin{remark}[Generalization]
It is possible to weaken the regularity conditions on the weight~$w$ in order to extend the results of \cref{prop:var_standard_HE} to a more general class of weight functions, in particular requiring only that $w$ be continuous at 0. Technical details are made precise in~\cref{appendix:generalizations_HE_var}.
\end{remark}

\section{Control variates for fluctuation formulas}
\label{sec:cv_methodology}
When used without modification,
the estimators~\eqref{eq:GK_estimator} and~\eqref{eq:HE_gen_estimator} presented from~\cref{sec:standard_formulas} can have a large variance.
In this section,
we construct and discuss some improved estimators based on control variates to calculate~$\rho$ in~\eqref{eq:tc_ip}:
\begin{equation}
    \rho = \ip{f,-\L^{-1}g} = \int_0^{+\infty} \E\bigl[f(X_0) g(X_t)\bigr] \, dt.
    \label{eq:poisson_rho_cv}
\end{equation}
In particular,
we consider three distinct ways of calculating~\eqref{eq:poisson_rho_cv} via a control variate approach.
These are all based on replacing,
in the estimators given in~\cref{sec:standard_formulas},
the function~$g$ or~$f$ or both by other ``smaller'' functions involving approximate solutions to Poisson equations.
By a slight abuse of terminology,
we call \emph{control variates} the functions used in place of $g$ or $f$ in the improved estimators.
The three approaches studied in this section are summarized below:
\begin{itemize}[leftmargin=*]
    \item
        \textbf{Forward control variate}:
        based on replacing by a control variate the function~$g$,
        which appears with a time dependence in the Green--Kubo formula~\eqref{eq:poisson_rho_cv}
        and in the related estimators~\eqref{eq:GK_estimator} and~\eqref{eq:HE_gen_estimator}.
        This approach is discussed in \cref{subsec:t_dep_cv}.
    \item
        \textbf{Adjoint control variate}:
        based on replacing by a control variate the function~$f$,
        which acts only in the initial condition in the Green--Kubo formula~\eqref{eq:poisson_rho_cv}
        and in the related estimators~\eqref{eq:GK_estimator} and~\eqref{eq:HE_gen_estimator}.
        This approach is discussed in \cref{subsec:zero_cost_cv}.
    \item
        \textbf{Combined control variates}:
        based on replacing both $f$ and $g$ by control variates, discussed in \cref{subsec:both_cv}.
\end{itemize}
Throughout this section,
we let $\F,G$ denote the solutions in~$L^2_0(\mu)$ to
\begin{subequations}
\begin{align}
    -\L^*\F = f, \label{eq:poisson_varphi_f} \\
    -\L G = g,
    \label{eq:poisson_varphi_g}
\end{align}
\end{subequations}
where $\L^*$ denotes the $L^2(\mu)$-adjoint of $\L$: for any test functions~$\varphi, \phi \in C^\infty_\mathrm{c}$,
\begin{equation}
    \notag
    \int_\mathcal{X} (\L\varphi)\phi \, d\mu = \int_\mathcal{X} \varphi(\L^*\phi) \, d\mu.
\end{equation}
Since $(\e^{t\L})^*$ = $\e^{t\L^*} \in \mathcal{B}(L^2_0(\mu))$,
decay estimates on the semigroup $\e^{t\L}$ such as \cref{as:semigroup_decay} immediately extend to its adjoint.
Thus, the well-posedness conditions discussed in \cref{cor:well_posedness_poisson} directly apply to~\eqref{eq:poisson_varphi_f}.
We also denote by $\psi_f^*$ and $\psi_g$ approximations of the solutions to~\eqref{eq:poisson_varphi_f} and~\eqref{eq:poisson_varphi_g}, respectively. These functions are assumed to belong to $L^2_0(\mu)$. Other integrability and regularity conditions are needed to make the variance of various estimators precise; see \crefrange{corollary:gk_time}{corollary:he_combined}.
The concrete construction of these approximate solutions will be discussed in the next section.
For notational convenience we also introduce
\begin{equation}
    \notag
    \zeta_w = \int_{0}^{1} (1-v)\w(v)^2 \, dv,
\end{equation}
for the factor associated with a weight function~$w$ that appears on the right-hand of the asymptotic variance formula~\eqref{eq:HE_var_prop_statement_asym}.
The three control variates studied in this section are summarized in~\cref{table:summary_cv}.

\begin{table}[htb!]
\footnotesize
\centering
\caption{Summary of the various control variates.
In all the formulas,
the functions $\psi_f^*,\psi_g$ denote respectively approximations of the solutions $F^*,G$ to the Poisson equations~$- \mathcal L^* F^* = f$ and $-\mathcal L G = g$.
The terms highlighted in dark green are the so-called ``static'' terms which can be evaluated via standard Monte Carlo methods,
while the terms highlighted in magenta are the only ones containing the unknown functions $F^*, G$.
These are all of the form $\ip{h, \mathcal L^{-1} \ell}$,
for appropriate functions $h, \ell$,
and so they can be approximated using the estimators of~\cref{sec:standard_formulas}.
}
\label{table:summary_cv}
\begin{tabular}{cc}
\toprule
\phantom{$\Big($}
    \textbf{Control variate} & \textbf{Formula for} $\rho = \ip{f,- \mathcal L^{-1} g}$
    \\ \midrule
    \phantom{$\Big($}
        {\bf Forward}
         &
         \(
         \displaystyle
         \rho = \textcolor{darkgreen}{\ip{f,\psi_g}} + \textcolor{magenta}{\ip{f, G - \psi_g}}
         \)
         \\
        Estimators~\eqref{eq:GK_cv_t_estimator} or~\eqref{eq:HE_cv_t_estimator} &
        Variance GK (\cref{corollary:gk_time}) $\propto \norm{f}^2 \mathcal E(G - \psi_g) \textcolor{red}{T}$ \\
         (\cref{subsec:t_dep_cv})
        & Variance HE (\cref{corollary:he_time}) $\propto \mathcal E(f) \mathcal E(G - \psi_g) $
         \\ \midrule
         \phantom{$\Big($}
             {\bf Adjoint}
         &
         \(
         \displaystyle
         \rho = \textcolor{darkgreen}{\ip{\psi_f^*, g}} + \textcolor{magenta}{\ip{f + \mathcal L^* \psi_f^*, G}}
         \)
         \\
        Estimators~\eqref{eq:GK_cv_0_estimator} or~\eqref{eq:HE_cv_0_estimator} &
        Variance GK (\cref{corollary:gk_zero}) $\propto \norm{f + \mathcal L^* \psi_f^*}^2 \mathcal E(G) \textcolor{red}{T}$ \\
        (\cref{subsec:zero_cost_cv})
        & Variance HE (\cref{corollary:he_zero}) $\propto \mathcal E(F^* - \psi_f^*) \mathcal E(G) $
         \\ \midrule
         \phantom{$\Big($}
        {\bf Combined}  &
        \(
        \displaystyle
        \rho = \textcolor{darkgreen}{\ip{\psi_f^*, g} + \ip{f, \psi_g} + \ip{\psi_f^*, \L\psi_g}} + \textcolor{magenta}{\ip{f + \mathcal L^* \psi_f^*, G - \psi_g}}
        \)
    \\
    Estimators~\eqref{eq:GK_cv_both_estimator} or~\eqref{eq:HE_cv_both_estimator} &
    Variance GK (\cref{corollary:gk_combined}) $\propto \norm{f + \mathcal L^* \psi_f^*}^2 \mathcal E(G - \psi_g) \textcolor{red}{T}$ \\
    (\cref{subsec:both_cv})
    & Variance HE (\cref{corollary:he_combined}) $\propto \mathcal E(F^* -\psi_f^*) \mathcal E(G - \psi_g) $
    \\ \bottomrule
\end{tabular}
\end{table}

\subsection{Forward control variate}
\label{subsec:t_dep_cv}
A simple calculation shows that
\begin{equation}
    \rho = \ip{f,-\L^{-1}g} = \ip{f,\psi_g} + \ip{f, G - \psi_g}.
    \label{eq:res_poisson_split}
\end{equation}
The first term on the right-hand side of \eqref{eq:res_poisson_split},
henceforth called the \emph{static} term,
corresponds to an approximation of~$\rho$ obtained from the approximate solution~$\psi_g$ to the Poisson equation,
and can be computed by resorting to standard Monte Carlo algorithms sampling $\mu$.
The second term is a correction term, which can be approximated using the estimators presented in~\cref{sec:standard_formulas}.
We next discuss estimators for this term and state some technical results which make precise the associated variances.

\paragraph{Static part}
In low dimensions,
the average $\ip{f,\psi_g}$ can be estimated via deterministic methods or classical i.i.d.\ Monte Carlo sampling.
When this is not possible, a Markov chain Monte Carlo (MCMC) approach can be employed:
\begin{equation}
    \notag
    \staticEst(f,\psi_g) = \frac{1}{TK}\sum_{k=1}^K \int_0^T f(X_t^k) \psi_g(X_t^k) \, dt.
\end{equation}
In any case, the approximation of $\ip{f,\psi_g}$ is much simpler than expressions involving correlations, and various variance reduction methods can be used; see \cite[Section 3.4]{lelievre2016}.

\paragraph{Green--Kubo} One possible estimator for the correction term in \eqref{eq:res_poisson_split} is given by the usual Green--Kubo estimator \eqref{eq:GK_estimator},
which gives:
\begin{equation}
    \GKcorr{g}{} = \frac{1}{K}\sum_{k=1}^K\int_0^T f(X^k_0) \Bigl(g(X^k_t) + \L\psi_g(X_t^k)\Bigr) \, dt.
    \label{eq:GK_cv_t_estimator}
\end{equation}
We next state a result on the variance of the improved estimator \eqref{eq:GK_cv_t_estimator}, a straightforward result directly obtained from \cref{prop:variance_standard_GK} by replacing $g$ with $g + \L\psi_g$.

\begin{corollary}
    [Variance of forward improved GK estimator]
    \label{corollary:gk_time}
Assume that $f,g$ satisfy the conditions given in \cref{prop:variance_standard_GK}, and that $\psi_g\in L^4(\mu)$ with $\L\psi_g\in L^2_0(\mu)$ and $\chi_{\psi_g}\in L^2(\mu)$. Then,
\begin{equation*}
        \lim_{T\to\infty} T^{-1}\Var\paren*{\GKcorr{g}{}} = \frac{2}{K} \norm{f}^2 \ip[\Big]{g + \L\psi_g, -\L^{-1}(g + \L\psi_g)}.
    \label{eq:gk_cor_statement_1}
\end{equation*}
\end{corollary}

Recall from the discussion after \cref{prop:variance_standard_GK}
that bounds on the resolvent are typically of order $\lambda^{-1}$ where $\lambda$ is the smallest nonzero eigenvalue of the operator $\L$. Thus, the $g$ term in the expression for the asymptotic variance is of order $\lambda^{-1}\norm{g + \L\psi_g}^2$. \cref{corollary:gk_time} thus suggests that the improved estimator is numerically useful when $\norm{g+\L\psi_g} \ll \norm{g}$, namely when the residual associated with the approximate solution of the Poisson equation is small enough.

\paragraph{Half-Einstein}
Another possible estimator for the correction term in \eqref{eq:res_poisson_split} is the half-Einstein estimator \eqref{eq:HE_gen_estimator}:
\begin{equation}
    \HEcorr{g}{} = \frac{1}{TK}\sum_{k=1}^K\int_0^T\int_0^t w\paren*{\frac{t-s}{T}}
        f(X^k_s) \Bigl(g(X^k_t) + \L\psi_g(X_t^k)\Bigr) \, ds \, dt.
    \label{eq:HE_cv_t_estimator}
\end{equation}
The variance of \eqref{eq:HE_cv_t_estimator} is made precise in the following corollary of \cref{prop:var_standard_HE}.

\begin{corollary}
    [Variance of forward improved HE estimator]
    \label{corollary:he_time}
    Assume that $f,g$ satisfy the conditions given in \cref{prop:var_standard_HE}, and that $\psi_g\in L^4(\mu)$ with $\L\psi_g\in L^4(\mu)$ and $\chi_{\psi_g}\in L^4(\mu)$. Then, \begin{equation*}
        \lim_{T\to+\infty} \Var\paren*{\HEcorr{g}{}} = \frac{4\zeta_w}{K}\ip{f, -\L^{-1}f}\ip{g + \L\psi_g, -\L^{-1}(g + \L\psi_g)}.
    \end{equation*}
\end{corollary}

Computing either \eqref{eq:GK_cv_t_estimator} or~\eqref{eq:HE_cv_t_estimator} requires evaluating the control variate $g + \L\psi_g$ at every step of the Monte Carlo simulation, which may be an expensive computation.
In the next section,
we present a control variate approach with a lower computational cost,
as it does not require evaluation at all steps.

\subsection{Adjoint control variate}
\label{subsec:zero_cost_cv}
It is possible to devise a control variate which need not be evaluated at all steps of the Monte Carlo run.
In particular, we construct a \emph{near-zero-cost} control variate (i.e., zero-cost aside from the computation of the approximate solution to the Poisson equation), as it is only evaluated once at initial time.
To this end, we first write \eqref{eq:poisson_rho_cv} in terms of the adjoint:
\begin{align}
    \notag
    \rho &= \ip{(-\L^{-1})^*f, g} \\
    \notag
    &= \ip{\psi_f^*, g} + \ip{\F - \psi_f^*, g} \\
    \label{eq:zero_cost}
    &= \ip{\psi_f^*, g} + \ip{f + \mathcal L^* \psi_f^*, - \mathcal L^{-1} g}.
\end{align}
As in the previous section,
the first term on the right-hand side is a static term that can be estimated with standard trajectory averages as
\begin{equation}
    \staticEst(g,\psi_f^*) = \frac{1}{TK}\sum_{k=1}^K \int_0^T g(X_t^k) \psi_f^*(X_t^k) \, dt,
    \label{eq:static_g_psif}
\end{equation}
while the second term can be approximated using the estimators presented in~\cref{sec:standard_formulas}, as made precise below.

\paragraph{Green--Kubo} The second term in~\eqref{eq:zero_cost} assumes the following estimator
\begin{equation}
    \GKcorr{f}{} = \frac{1}{K}\sum_{k=1}^K\int_0^T \paren*{f(X^k_0) + \L^*\psi_f^*(X_0^k)}g(X^k_t) \, dt.
    \label{eq:GK_cv_0_estimator}
\end{equation}
Note that this estimator requires the computation of the control variate $f + \L^*\psi_f^*$ only at the initial time,
eliminating the need to evaluate the control variate at every step of the Monte Carlo simulation.

\begin{corollary}
    [Variance of adjoint improved GK estimator]
    \label{corollary:gk_zero}
    Assume that $f,g$ satisfy the conditions given in \cref{prop:variance_standard_GK}, and that $\L^*\psi_f^*\in L^4(\mu)$. Then,
    \begin{equation}
        \notag
        \lim_{T\to+\infty} T^{-1}\Var\paren*{\GKcorr{f}{}} = \frac{2}{K} \norm{f + \L^*\psi_f^*}^2 \ip[\big]{g,-\L^{-1}g}.
    \end{equation}
\end{corollary}

Note that $\L^*\psi_f^*$ has average 0 with respect to $\mu$:
\begin{equation}
    \notag
    \int_\X \L^*\psi^*_f \, d\mu = \int_\X \psi^*_f (\L\ind) \, d\mu = 0,
\end{equation}
so that the condition $\L^*\psi_f^* \in L^4_0(\mu)$ from \cref{prop:variance_standard_GK} reduces to $\L^*\psi_f^* \in L^4(\mu)$.

\paragraph{Half-Einstein} The second term in \eqref{eq:zero_cost} assumes the following estimator
\begin{equation}
    \HEcorr{f}{} = \frac{1}{TK}\sum_{k=1}^K\int_0^T\int_0^t w\left( \frac{t-s}{T} \right)
    \Bigl(f\bigl(X^k_s\bigr) + \L^*\psi_f^*\bigl(X_s^k\bigr)\Bigr) g\bigl(X^k_t\bigr) \, ds \, dt.
    \label{eq:HE_cv_0_estimator}
\end{equation}

\begin{corollary}
    [Variance of adjoint improved HE estimator]
    \label{corollary:he_zero}
    Assume that $f,g$ satisfy the conditions given in \cref{prop:var_standard_HE}, and that 
$\L^*\psi_f^*\in L^4(\mu)$, $\L^{-1}\L^*\psi_f^*\in L^4(\mu)$ and $\chi_{\L^{-1}\L^*\psi_f^*}\in L^4(\mu)$. Then,
\begin{equation*}
        \lim_{T\to+\infty} \Var\paren*{\HEcorr{f}{}} = \frac{4\zeta_w}{K}
        \ip[\Big]{f + \L^*\psi_f^*, -\L^{-1}(f + \L^*\psi_f^*)}\ip[\Big]{g, -\L^{-1}g}.
    \end{equation*}
\end{corollary}

\subsection{Combining both control variates}
\label{subsec:both_cv}
By obtaining approximate solutions to both Poisson equations, we can apply both control variates simultaneously.
By writing $-\L^{-1}g = \psi_g + (G - \psi_g)$ in \eqref{eq:zero_cost},
we obtain
\begin{align}
    \notag
\rho
    &= \ip{\psi_f^*, g} + \ip{f + \mathcal L^* \psi_f^*, G} \\
    \notag
    &= \ip{\psi_f^*, g} + \ip{f + \mathcal L^* \psi_f^*, \psi_g} + \ip{f + \mathcal L^* \psi_f^*, G - \psi_g} \\
    &= \ip{\psi_f^*, g} + \ip{f, \psi_g} + \ip{\psi_f^*, \L\psi_g} + \ip{f + \mathcal L^* \psi_f^*, G - \psi_g}.
    \notag
\end{align}
The first three terms can be estimated directly using trajectory averages,
while the last term can be approximated using the estimators in~\cref{sec:standard_formulas}.

The GK improved estimator reads in this context
\begin{equation}
\begin{aligned}
    \GKcorr{f}{g} =
    \frac{1}{K}\sum_{k=1}^K\int_0^T \Bigl(f(X^k_0) + \L^*\psi_f^*(X_0^k)\Bigr)\Bigl(g(X^k_t) + \L\psi_g(X_t^k)\Bigr) \, dt.
    \label{eq:GK_cv_both_estimator}
\end{aligned}
\end{equation}

\begin{corollary}
    [Variance of combined improved GK estimator]
    \label{corollary:gk_combined}
Assume that~$f,g$ satisfy the conditions given in \cref{prop:variance_standard_GK}, and that in addition the conditions from \cref{corollary:gk_time,corollary:gk_zero} hold. Then,
\begin{equation}
    \notag
    \lim_{T\to+\infty} T^{-1}\Var\paren*{\GKcorr{f}{g}} = \frac{2}{K}\norm{f + \L^*\psi_f^*}^2 \ip[\Big]{g + \L\psi_g, -\L^{-1}(g + \L\psi_g)}.
\end{equation}
\end{corollary}
The HE improved estimator reads
\begin{equation}
\begin{aligned}
    \HEcorr{f}{g} &= \frac{1}{TK}\sum_{k=1}^K\int_0^T\int_0^t w\paren*{\frac{t-s}{T}}
    \Bigl(f(X^k_s) + \L^*\psi_f^*(X_s^k)\Bigr)\Bigl(g(X^k_t) + \L\psi_g(X_t^k)\Bigr) \, ds \, dt.
    \label{eq:HE_cv_both_estimator}
\end{aligned}
\end{equation}

\begin{corollary}
    [Variance of combined improved HE estimator]
    \label{corollary:he_combined}
Assume that $f,g$ satisfy the conditions given in \cref{prop:var_standard_HE}, and that in addition the conditions from \cref{corollary:he_time,corollary:he_zero} hold. Then, \begin{equation*}
        \lim_{T\to+\infty} \Var\paren*{\HEcorr{f}{g}}
        = \frac{4\zeta_w}{K}\ip[\Big]{f + \L^*\psi_f^*, -\L^{-1}(f + \L^*\psi_f^*)}\ip[\Big]{g + \L\psi_g, -\L^{-1}(g + \L\psi_g)}.
    \end{equation*}
\end{corollary}

 \section{Applications}
\label{sec:numerics}
The aim of this section is to apply the control variate methods discussed in \cref{sec:cv_methodology} to different systems in order to assess the effectiveness of the various estimators. We start by discussing in \cref{subsec:PINNs} the general approach for solving Poisson equations using neural networks. We then apply the method to the underdamped Langevin dynamics in \cref{subsec:lang_application}, and to multiscale SDEs in \cref{subsec:multiscale}.

\subsection{Approximate solutions to Poisson equations with neural networks}
\label{subsec:PINNs}
We now discuss the overall strategy for approximating the solution to Poisson equations using PINNs. In particular, we discuss the loss function used and its discretization, and outline the training strategy common to all examples in this section. More specific architectural choices, such as activation functions, featurization strategies and network topology are overviewed in the respective numerical application sections to come.

\paragraph{Loss function} Let us consider the following Poisson equation on $\X$:
\begin{equation}
    \notag
    -\L\phi = f,
\end{equation}
with $f \colon \X \to \R$. We denote by $\phi^\mathrm{NN}_\theta(x)$ the neural network approximation of $\phi$ parametrized by a set of parameters $\theta$. The pointwise loss function we consider is the mean squared error of the residual $r_\theta(x) = \L\phi^\mathrm{NN}_\theta(x) + f(x)$. At the continuous level, the loss function is then
\begin{equation}
    \Loss(\theta) = \int_\mathcal{X} \left(\L\phi^\mathrm{NN}_\theta(x) + f(x)\right)^2 \mu(dx).
    \label{eq:loss_continuous}
\end{equation}
In practice, the integral in \eqref{eq:loss_continuous} is approximated with a sum over the data points $\{\x^n\}_{n=1}^N \in \R^d$ at which the functions are evaluated. For the situations we consider, the points are directly sampled from $\mu$ via rejection sampling, or by drawing random variables according to Gaussian distributions. Additionally, we consider centered finite differences to approximate differential operators in $\L$ acting on $\phi$, as automatic differentiation becomes too cumbersome for derivatives of order 3 and higher. Thus, automatic differentiation is only used to take the gradient of the network with respect to the parameters during training.
The finite difference discretizations are implemented with optimal stepsize (see~\cite{sauer2011}), namely~$\sqrt[3]{\epsilon_{\rm mach}}$ for central first derivative approximations, and~$\sqrt[4]{\epsilon_{\rm mach}}$ for central second derivative approximations, where~$\epsilon_{\rm mach}$ denotes the machine epsilon.
To this end, let~$\L^\Delta$ denote the discretization of~$\L$. Then, our loss function is of the form
\begin{equation}
    \notag
    \Loss(\theta) = \frac{1}{N}\sum_{n=1}^N \abs{\L^\Delta\phi^\mathrm{NN}_\theta(\x^n) + f(\x^n)}^2.
\end{equation}
A similar loss is considered to solve the Poisson equation $-\L^*\phi^* = f$.

\paragraph{Training strategy} The same overall strategy was used to train the networks for both examples here presented.
The computational cost associated with sampling the data points is negligible, allowing us to compute a stochastic approximation of the gradient with a batch of~$N$ new points at every step. Since the usual overfitting risks associated with a finite training set therefore do not apply, we performed a fixed number of training steps and did not implement early stopping. Training was performed using the Adam optimizer \cite{kingma2017}.

The network topologies considered were simple networks with two dense hidden layers, and the same activation function throughout, in particular $\tanh$ for both examples. Other activation functions were explored, e.g., regularized ReLU-type functions such as GELU \cite{hendrycks2023} were observed to perform marginally better than $\tanh$, however at an increased computational cost. The addition of the second hidden layer significantly improves the quality of training, while still retaining a sufficiently low cost. Based on numerical observations from experimental runs, width has a more significant impact on the quality of training than depth for a fixed number of network parameters.

All training parameters are made precise in \cref{table:training_params} for the Langevin dynamics and multiscale dynamics examples.
Let us emphasize that a perfect training is anyway not necessary, as approximate solutions to the Poisson equation at hand are sufficient to obtain variance reduction.

\begin{table}[htb]
    \footnotesize
    \begin{center}
    \caption{Training parameters for training the network for the two-dimensional Langevin dynamics and multiscale dynamics systems.}
    \begin{tabular}{ccc}
    \toprule
        \textbf{Parameter} & \textbf{Langevin} & \textbf{Multiscale} \\
        \midrule
        Adam learning rate & 0.002 & 0.002 \\
        Batch size $N$ & 500 & 1000 \\
        Number of training steps & 2000 & 1000 \\
        \bottomrule
    \end{tabular}
    \label{table:training_params}
    \end{center}
\end{table}

\begin{figure}[htb!]
    \centering
    \includegraphics[width=0.8\linewidth]{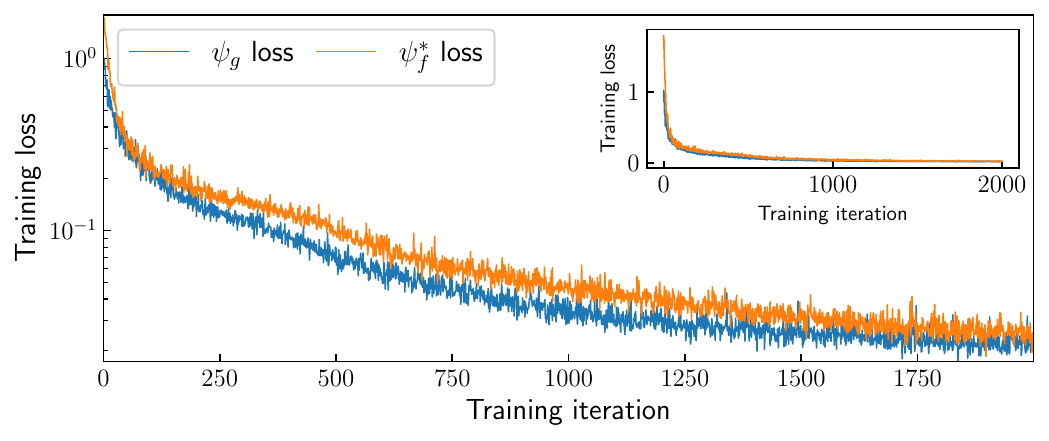}
    \caption{Training iteration steps vs training loss for the 2D Langevin dynamics for both $\psi_g$ and $\psi_f^*$ approximate solutions. Top: linear scale. Bottom: logarithmic y-axis.}
    \label{fig:training_loss}
\end{figure}

The same training routine was used to approximate $\phi$ and $\phi^*$. \cref{fig:training_loss} shows the training loss as a function of training steps for both approximations for the Langevin dynamics example.

\subsection{Application to the underdamped Langevin dynamics}
\label{subsec:lang_application}
We apply in this section the methodology to compute the mobility of the two-dimensional underdamped Langevin dynamics on the torus. We start by giving a brief description of the system in \cref{subsubsec:description_2d-lang}, followed by a presentation of the numerical results in \cref{subsubsec:num_results_2d-lang}.

\subsubsection{Description of the system}
\label{subsubsec:description_2d-lang}
Let us briefly describe the system under consideration. The two-dimensional Langevin dynamics evolves the positions $q = (q_1,q_2)\in \T^2$ and momenta $p = (p_1,p_2)\in \R^2$ (so that $\X = \T^2 \times \R^2$) according to the following SDE:
\begin{align}
\begin{split}
\begin{cases}
    dq_t = M^{-1}p_t \, dt, \\
    dp_t = -\nabla V(q_t) \, dt - \gamma M^{-1}p_t \, dt + \sqrt{\dfrac{2\gamma}{\beta}} \, dW_t,
    \label{eq:langevin_dynamics}
\end{cases}
\end{split}
\end{align}
where $M \in \R^{2\times 2}$ denotes the mass matrix, $\gamma>0$ the friction coefficient, $\beta>0$ the inverse temperature and $V$ a smooth periodic potential energy function. The generator of the process~\eqref{eq:langevin_dynamics} is
\begin{equation}
    \label{eq:lang_generator}
    \L = p^\t M^{-1}\nabla_q - \nabla V^\t \nabla_p + \gamma(-p^\t M^{-1}\nabla_p + \beta^{-1}\Delta_p).
\end{equation}
The dynamics \eqref{eq:langevin_dynamics} is ergodic with respect to the Boltzmann--Gibbs probability measure:
\begin{equation}
    \notag
    \mu(dq \, dp) = \frac{1}{Z}\e^{-\beta H(q,p)} \, dq \, dp, \qquad Z = \int_{\T^2\times \R^2}\e^{-\beta H(q,p)} dq \, dp.
\end{equation}
The transport coefficient of interest is the mobility in the direction $\evec$. It is defined as the proportionality constant between the magnitude $\eta$ of an exerted constant force proportional to $\evec$, and the induced drift velocity $\E(f_\evec)$ in this direction, with $f_\evec = \evec^\t M^{-1}p$. The relationship is linear for~$\eta\to 0$, as dictated by linear response theory; see for instance \cite[Chapter 8]{chandler1987} and \cite[Section 5]{lelievre2016}. Through an equilibrium formulation of the linear response based on velocity autocorrelations \cite{resibois1977}, it can in particular be written in terms of the solution to a Poisson equation:
\begin{equation}
    \notag
    \rho = \beta\ip{\phi, f_\evec}, \qquad -\L\phi(q,p) = f_\evec(q,p).
\end{equation}
Note that the mobility is proportional up to factor $\beta$ to the self-diffusion coefficient through Einstein's relation; see \cite{rodenhausen1989,latorre2013}.
The points $\{\x^n\}_{n=1}^N = (q_1,q_2,p_1,p_2)_n \in \R^4$ are i.i.d.\ according to the target measure $\mu$, sampled in practice using rejection sampling for the positions, and by sampling a Gaussian distribution for the momenta.

\subsubsection{Numerical results}
\label{subsubsec:num_results_2d-lang}
We consider the following nonseperable potential energy function:
\begin{equation}
    \notag
    V(q_1, q_2) = -\frac{\cos(2q_1) + \cos(q_2)}{2} - \delta\cos(q_1)\cos(q_2),
\end{equation}
with $\delta\in\R$ the degree of nonseparability. The numerical results presented here correspond to the value $\delta = 0.5$. Additionally, we set the physical parameters $\beta = \gamma = 1$ and $M = \Id_2$.

\paragraph{Structure of the network} Aside from the general training strategy outlined in \cref{subsec:PINNs}, the network considered in this example consists of two dense layers with intermediate dimension 15, with four input nodes and one output. Additionally, we employ a transformation on the inputs in the form of a functional ``featurization layer'': the input $(q_1,q_2,p_1,p_2)$ is transformed into a 7-node output by the functions
\begin{equation}
    \label{eq:featurization}
    \mathrm{featurization}\paren*{q_1,q_2,p_1,p_2} = \paren*{\sin(q_1), \cos(q_1), \sin(q_2), \cos(q_2), p_1, p_2, \frac{p_1^2 + p_2^2}{2}},
\end{equation}
which is then fed to the first trainable layer; see \cref{appendix:nn_architecture} for a discussion on the featurization and \cref{fig:2d-lang_network_both} for an illustration of the network topology. The periodic nature of the positions is encoded by the sines and cosines in the featurization layer, while the momenta are directly included. Lastly, the kinetic energy term was observed to be a meaningful feature in improving training, and is thus considered.

\paragraph{Monte--Carlo simulations} To numerically integrate \eqref{eq:langevin_dynamics}, we used the BAOAB splitting scheme \cite{leimkuhler2013}, with an integration time of~$T=5$ for the Green--Kubo estimator and~$T=10$ for the half-Einstein estimator, with timestep $\Delta t = 0.01$ for both. We ran $K=1000$ independent realizations of the dynamics in order to empirically estimate the variance. Initial conditions for each realization were also sampled from $\mu$ using rejection sampling for the positions.

\begin{figure}[ht]
    \centering
    \includegraphics[width=\linewidth]{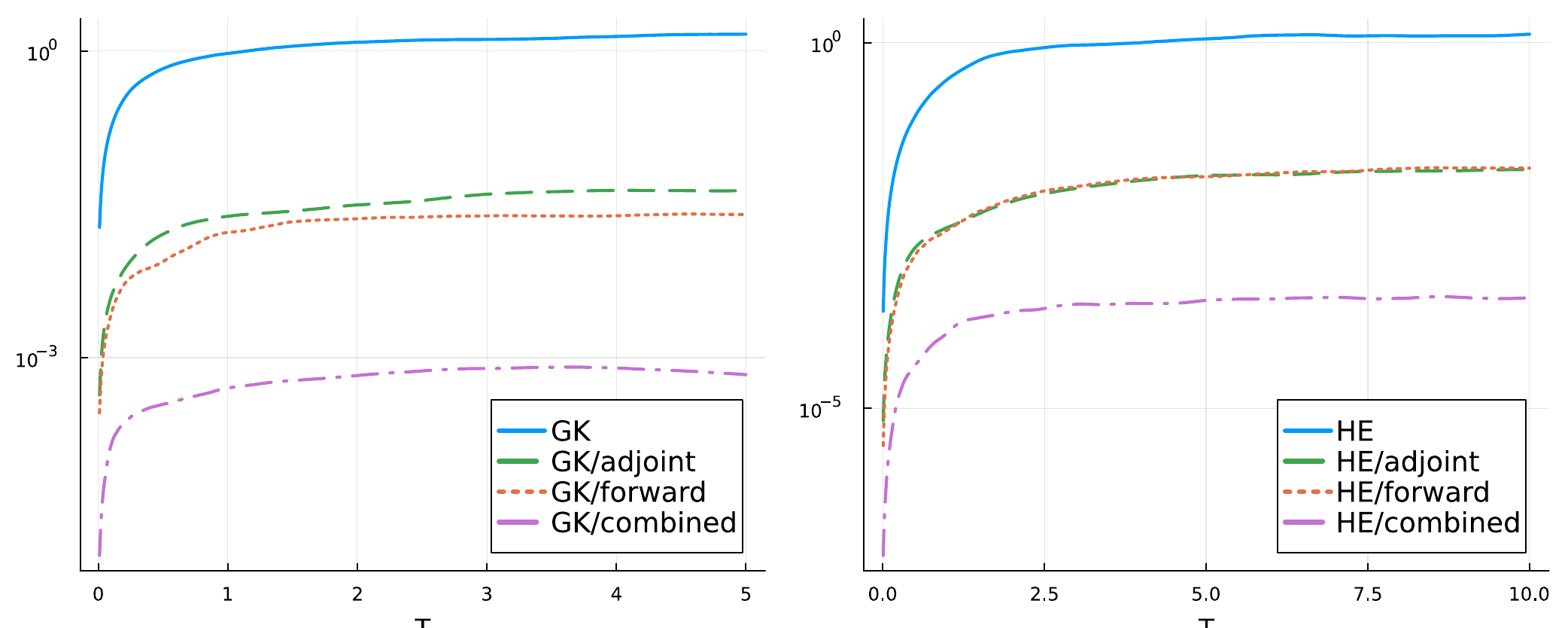}
    \caption{Asymptotic variance over time for all estimators for the two-dimensional Langevin dynamics. Left: Green--Kubo estimators. Right: half-Einstein estimators. For Green--Kubo, the asymptotic variance is defined as the variance divided by time, while the two coincide for half-Einstein.}
    \label{fig:2d-lang_variances}
\end{figure}
Simulation results are illustrated in \cref{fig:2d-lang_variances}, which shows the variance as a function of time for all 8 estimators considered. A plot of the asymptotic variance is also shown, which corresponds to the variance divided by time for Green--Kubo estimators, while it coincides with the variance for half-Einstein estimators. A summary of the results is compiled in \cref{tab:2d-lang_normalized}, which provides a quantitative metric for assessing the viability of each estimator. In particular, it shows the computational runtime, the variance and the \emph{cost} (defined as the product of runtime and variance) for each estimator.

The numerical results suggest that the use of both control variates (for both GK and HE) is the best option according to the cost. This, however, requires a decently trained network, as the forward adjoint is computed every step of the Monte Carlo simulation, thus it must provide enough variance reduction in order to offset the additional computational cost. The Green--Kubo adjoint estimator, on the other hand, is a (near) ``zero-cost'' control variate, as it is only evaluated once. This suggests that it is much more forgiving of nonoptimal training runs.

\begin{table}[tbh]
\footnotesize
    \captionsetup{position=top}
    \caption{Computational runtimes required to compute the self-diffusion coefficient for the two-dimensional Langevin dynamics and associated variances at time $T=5$, based on $K=1000$ realizations of the dynamics. The cost is defined as the runtime times the variance.}\label{tab:2d-lang_normalized}
\begin{center}
    \subfloat[Green--Kubo data, normalized in terms of the standard GK estimator runtime of $0.157$ seconds, and associated variance $7.331$.]{\label{tab:2d-lang_normalized_gk}
        \begin{tabular}{cccc}
\toprule
\textbf{Estimator}       & \textbf{Runtime} & \textbf{Variance at $T=5$} & \textbf{Cost} \\
\midrule
GK~\eqref{eq:GK_estimator} & $1$  & $1$ & $1$ \\
GK/adjoint~\eqref{eq:GK_cv_0_estimator}  & $0.957$  & $0.0294$ & $0.0281$ \\
GK/forward~\eqref{eq:GK_cv_t_estimator}  & $22.9$ & $0.0172$ & $0.394$ \\
GK/combined~\eqref{eq:GK_cv_both_estimator} & $25.1$ & $0.000465$ & $0.0117$ \\
\bottomrule
\end{tabular}}

    \subfloat[Half-Einstein data, normalized in terms of the standard HE estimator runtime of $0.297$ seconds, and associated variance $1.313$.]{\label{tab:2d-lang_normalized_he}
        \begin{tabular}{cccc}
\toprule
\textbf{Estimator}       & \textbf{Runtime} & \textbf{Variance at $T=10$} & \textbf{Cost} \\
\midrule
HE~\eqref{eq:HE_gen_estimator}           & $1$ & $1$ & $1$ \\
HE/adjoint~\eqref{eq:HE_cv_0_estimator}  & $23.9$ & $0.0140$ & $0.336$ \\
HE/forward~\eqref{eq:HE_cv_t_estimator}  & $22.6$ & $0.0147$ & $0.332$ \\
HE/combined~\eqref{eq:HE_cv_both_estimator}  & $44.6$ & $0.000246$ & $0.0109$ \\
\bottomrule
\end{tabular}}
\end{center}
\end{table}

\subsection{Application to multiscale SDEs}
\label{subsec:multiscale}
The aim of this subsection is to present an application of the control variate approaches studied in~\cref{sec:cv_methodology},
with slight adjustments,
to fast/slow systems of stochastic differential equations.
We first describe briefly,
in~\cref{ssub:homogenization_for_multiscale_SDEs},
the method of homogenization for multiscale SDEs.
Then, in~\cref{ssub:control_variate_approaches_for_calculating_the_effective_coefficients},
we explain how the control variate approaches from~\cref{sec:cv_methodology} can be applied for the calculation of the effective coefficients of the homogenized equation.
Finally, in~\cref{ssub:numerical_example},
we illustrate the performance of the method on a concrete example.

\subsubsection{Homogenization for multiscale SDEs}\label{ssub:homogenization_for_multiscale_SDEs}
In this section,
we consider the following multiscale system with state space $\real^{d_x} \times \real^{d_y}$,
for some small parameter~$\varepsilon>0$:
\begin{equation}
    \label{eq:multiscale}
    \left\{
    \begin{aligned}
        \d X^{\varepsilon}_t &= \frac{1}{\varepsilon} f(X^{\varepsilon}_t, Y^{\varepsilon}_t) \, \d t, \\
        \d Y^{\varepsilon}_t &= \frac{1}{\varepsilon^2} g(Y^{\varepsilon}_t) \, \d t + \frac{1}{\varepsilon} \alpha(Y^{\varepsilon}_t) \, \d W_t,
    \end{aligned}
    \right.
\end{equation}
where $f\colon \real^{d_x} \times \real^{d_y} \to \real^{d_x}$, $g\colon \real^{d_y} \to \real^{d_y}$
and $\alpha\colon \real^{d_y} \to \real^{d_y \times d_w}$
are smooth functions and~$(W_t)_{t\geq 0}$
is a standard Brownian motion in $\real^{d_w}$.
This simple setting,
or a variation thereof where the drift and diffusion coefficients of the fast process also include lower order terms allowed to depend on~$X^{\varepsilon}_t$,
is often considered in the literature,
for example in numerical works~\cite{vandeneijnden2003,e2005}.
The infinitesimal generator associated with the fast process~$Y^{\varepsilon}$ is given by~$\varepsilon^{-2} \mathcal L$,
where
\[
    \mathcal L = g(y) \cdot \nabla_y + \frac{1}{2} \alpha(y) \alpha(y)^\t : \nabla_y^2.
\]
By the theory of homogenization for multiscale SDEs,
if the fast process admits a unique ergodic distribution~$\mu \in \mathcal P(\real^{d_x})$,
and if the following centering condition is satisfied,
\[
    \forall x \in \real^{d_x}, \qquad
    \int_{\real^{d_x}} f(x, y) \, \mu(\d y) = 0,
\]
then in the limit $\varepsilon \to 0$ the slow process $(X^{\varepsilon}_t)_{t \geq 0}$ converges weakly in $C[0, T]$ to the solution of the homogenized SDE:
\begin{equation}
    \label{eq:homogenized_equation}
    \d \overline X = F(\overline X) + A(\overline X) \, \d V_t.
\end{equation}
Here $V_t$ is a standard Brownian motion in~$\real^{d_x}$ and the so-called effective drift and diffusion coefficients are given by
\[
    F(x) = \int_{\real^{d_y}} \Bigl( f(x, y) \cdot \nabla_x \Bigr) \Phi(x, y) \, \mu(\d y) \in \real^{d_x},
    \qquad
    A(x) A(x)^\t = A^0(x) + A^0(x)^\t \in \real^{d_x \times d_x},
\]
where $\Phi$ is a solution to
\(
- \mathcal L^x \Phi(x, y) = f(x, y)
\)
and
\[
    A^0(x) = \int_{\real^{d_y}} f(x, y) \otimes \Phi(x, y) \, \mu^x(\d y).
\]
We refer to~\cite[Chapter 18]{pavliotis2008} for a rigorous proof in a simplified setting,
and to~\cite{pardoux2001,pardoux2003} for extensions and additional results.
Since formally $\nabla_x \mathcal L^{-1} = \mathcal L^{-1} \nabla_x$,
given that $\mathcal L$ only includes derivatives with respect to~$y$,
the coefficients of the homogenized equation~\eqref{eq:homogenized_equation} are of the form~\eqref{eq:tc_ip}.
More precisely, the components of the coefficients~$F$ and~$A^0$
can be written as
\begin{equation}
    \label{eq:homogenized_coefficients}
    F_i(x) = \sum_{j=1}^{d_x} \Bigl\langle f_j(x, \placeholder) , - \mathcal L^{-1} \mathfrak h_{ij}(x, \placeholder)  \Bigr\rangle,
    \qquad
    A^0_{ij}(x) = \Bigl\langle f_i(x, \placeholder) , - \mathcal L^{-1} f_j(x, \placeholder) \Bigr\rangle,
    \qquad
\end{equation}
where $\mathfrak h_{ij}(x, \placeholder) = \partial_{x_j} f_i(x, \placeholder)$,
and with $\ip{\placeholder, \placeholder}$ the inner product with respect to~$L^2(\mu)$,
with~$x$ viewed as a fixed parameter.
It follows from~\eqref{eq:homogenized_coefficients} that
the homogenized coefficients can be approximated using estimators based on the Green--Kubo formula~\eqref{eq:GK},
such as those studied in~\cref{sec:standard_formulas}.
This constitutes the basic idea behind the so-called \emph{heterogeneous multiscale method} for multiscale SDEs~\cite{vandeneijnden2003,e2005}.
The variance reduction method we describe in \cref{ssub:control_variate_approaches_for_calculating_the_effective_coefficients}
is based on the specific estimators given in~\cref{sec:cv_methodology},
but we note that it can easily be adapted to the estimators used within the heterogeneous multiscale method,
such as those given in~\cite[Section~3.1]{e2005}.

\subsubsection{Control variate approaches for calculating the effective coefficients}\label{ssub:control_variate_approaches_for_calculating_the_effective_coefficients}
In this section,
we describe how the control variate approaches described in~\cref{sec:cv_methodology} can be employed for the calculation of the coefficients of the homogenized equation~\eqref{eq:homogenized_equation}.
The main difference with the other examples considered so far is the presence of an~$x$-dependence in~\eqref{eq:homogenized_coefficients}
and the associated Poisson equations
\begin{equation}
    \label{eq:poisson_with_x}
    - \mathcal L \mathscr H_{ij}(x, y) = \mathfrak h_{ij}(x, y),
    \qquad
    - \mathcal L \mathscr F_{i}(x, y) = f_{i}(x, y),
\end{equation}
as well as in the associated adjoint equation
\begin{equation}
    \label{eq:poisson_adjoint_x}
    - \mathcal L^* \mathscr F^*_i(x, y) = f_i(x, y).
\end{equation}
The control variate methods described in~\cref{sec:cv_methodology} require approximate solutions
to~\eqref{eq:poisson_with_x} or~\eqref{eq:poisson_adjoint_x} or both.
We next describe two possible approaches for constructing these approximate solutions:
\begin{itemize}[leftmargin=*]
    \item
        A first approach would be to solve the Poisson equations
        -- \eqref{eq:poisson_with_x} or~\eqref{eq:poisson_adjoint_x} or both, depending on which method from~\cref{sec:cv_methodology} is considered --
        for each fixed~$x$ separately as needed.
        Combined with a classical time-stepping algorithm for the homogenized equation~\eqref{eq:homogenized_equation},
        this would lead to a numerical method where a number of Poisson PDEs in dimension~$d_y$  must be solved at each time step,
        much like the method studied in~\cite{abdulle2017}.

    \item
        An alternative approach,
        used in~\cref{ssub:numerical_example},
        is to solve the Poisson equations~\eqref{eq:poisson_with_x} for all $x \in \real^{d_x}$
        and all vector components simultaneously.
        This is not simple to achieve using traditional PDE solvers,
        but simple to implement using neural networks.
        For example,
        in order to numerically solve the second equation in~\eqref{eq:poisson_with_x},
        we search for an approximate vector-valued solution~$\mathscr F^{\rm NN}(x, y; \theta)$
        within a class of functions parametrized by a neural network with parameters~$\theta$.
        The vector of parameters~$\theta$ is found by minimizing the following loss function
        \begin{equation}
            \label{eq:loss_multiscale}
            \mathscr L(\theta) = \int_{\real^{d_x} \times \real^{d_y}}
            \sum_{i=1}^{d_x} \Bigl\lvert f_i(x, y) + \mathcal L \mathscr F^{\rm NN}_i(x, y; \theta) \Bigr\rvert^2 \, \nu(\d x) \, \mu(\d y),
        \end{equation}
        where $\mu$ is, as before,
        the invariant probability measure associated with the fast processes in~\eqref{eq:multiscale},
        and $\nu \in \mathcal P(\real^{d_x})$ is a probability measure on~$\real^{d_x}$ that
        must be fixed a priori as part of the method.
        A natural choice, when~$\real^{d_x}$ is a compact state space,
        is to let~$\nu$ be the uniform distribution on~$\real^{d_x}$.
        When~$\real^{d_x}$ is non-compact,
        the choice of $\nu$ determines the region of~$\real^{d_x}$
        where we wish~$\mathscr F^{\rm NN}(x, \placeholder; \mathbf c)$ to be a good approximation of the exact solution~$\mathscr F(x, \placeholder)$,
        and where variance reduction can be expected as a result.
        In practice, the loss function~\eqref{eq:loss_multiscale} of course needs to be discretized,
        for example by estimating the integral via Monte Carlo sampling.
\end{itemize}
To conclude this subsection,
let us mention that it is not necessary to solve the first equation in~\eqref{eq:poisson_with_x} explicitly.
Indeed, the function $\mathscr H_{ij}(x, y)$ is the $x$-gradient of~$\mathscr F_{i}(x, y)$,
and so a numerical approximation of $\mathscr H_{ij}(x, y)$ can be obtained by taking the $x$-gradient of $\mathscr F^{\rm NN}_i(x, y; \theta)$.

\subsubsection{Numerical example}\label{ssub:numerical_example}

In this section,
we illustrate the performance of the control variate strategies presented in~\cref{sec:cv_methodology},
when the approximate solutions to Poisson equations are calculated simultaneously for all~$x$ by a neural network approach.
We consider the following simple example of a multiscale SDE,
taken from~\cite[Section~11.7.7]{pavliotis2008},
to which we refer for additional context:
\begin{equation}
    \label{eq:example_multiscale}
    \left\{
        \begin{aligned}
            \d X_1 &= \frac{1}{\varepsilon} Y_1 \, \d t \, , \\
            \d X_2 &= \frac{1}{\varepsilon} Y_2 \, \d t \, , \\
            \d X_3 &= \frac{1}{\varepsilon} (X_1 Y_2 - X_2 Y_1) \, \d t \, , \\
            \d Y_1 &= - \frac{Y_1}{\varepsilon^2} \, \d t - \alpha \frac{Y_2}{\varepsilon^2} \, \d t + \frac{1}{\varepsilon} \d W^1_t \, ,\\
            \d Y_2 &= - \frac{Y_2}{\varepsilon^2} \, \d t + \alpha \frac{Y_1}{\varepsilon^2} \, \d t + \frac{1}{\varepsilon} \d W^2_t \, ,
        \end{aligned}
    \right.
\end{equation}
where $\alpha > 0$ and $W^1, W^2$ are standard independent Brownian motions.
In this example, $d_x = 3$ and $d_y = 2$,
so that the total dimension is equal to 5.
An application of the homogenization result in~\cref{ssub:homogenization_for_multiscale_SDEs} leads to the following homogenized equation:
\begin{equation}
    \notag
    \left\{
        \begin{aligned}
            \d \overline X_1 &= \frac{1}{1 + \alpha^2} \bigl(\d V^1_t - \alpha \, \d V^2_t\bigr) \, , \\
            \d \overline X_2 &= \frac{1}{1 + \alpha^2} \bigl(\d V^2_t + \alpha \, \d V^1_t\bigr) \, , \\
            \d \overline X_3 &= \frac{1}{1 + \alpha^2} \, \d t + \frac{1}{1 + \alpha^2} \Bigl( \bigl(\alpha \overline X_1 - \overline X_2\bigr) \, \d V^1_t + \bigl(\alpha \overline X_2 + \overline X_1\bigr) \, \d V^2_t \Bigr),
        \end{aligned}
    \right.
\end{equation}
where $(V^1)_{t\geq 0}$ and $(V^2)_{t\geq 0}$ are independent Brownian motions.
We remark that an additional term appears in the equation for~$\overline X_3$,
compared to what might be expected from a heuristic approach,
due to the L\'evy area correction; see~\cite[Section 11.7.7]{pavliotis2008} for details on this point.
In this numerical experiment,
we fix $\alpha = 1$ and compare the performance of the estimators given in~\cref{sec:standard_formulas,sec:cv_methodology} for the computation of the homogenized coefficients.
To this end,
we solve the Poisson equations (both the forward equations involving~$\mathcal L$,
and the adjoint equations involving~$\mathcal L^*$),
simultaneously for all~$x$.
Furthermore, we solve the vector-valued Poisson equations simultaneously for all the components.
For example, in order to solve the Poisson equations
\begin{equation}
    \label{eq:poisson_drift_multiscale}
    - \mathcal L \mathscr F(x, \placeholder) = f(x, \placeholder),
\end{equation}
with $f(x, y) \in \real^3$ the drift of the slow process,
we use a neural network consisting of two dense hidden layers of dimension 12. The network comprises 5 inputs,
corresponding to the five variables~$(x_1, x_2, x_3, y_1, y_2)$,
and three outputs,
corresponding to the three components of the vector valued solution; see \cref{fig:2d-lang_network_both} for an illustration of the network architecture.
For the measure~$\nu$,
we take the uniform distribution over the cube $[-4, 4]^3$,
and for~$\mu$ we take the invariant measure of the fast processes.
Since $(Y_1, Y_2)$ is an Ornstein--Uhlenbeck process,
it follows from~\cite[Chapter~9]{lorenzi2007} that
the latter probability measure is given by the normal distribution~$\mathcal N\left(0,\frac{1}{2}\I_2\right)$.
The other training parameters are given in~\cref{table:training_params}.

Having calculated approximate solutions to~\eqref{eq:poisson_drift_multiscale} and its adjoint counterpart,
we then fix a value for $(X_1, X_2, X_3) = (-0.0057, 1.73, -1.04)$,
randomly drawn from the standard normal distribution in dimension 3,
and evaluate the quality of variance reduction for the estimators in~\cref{sec:cv_methodology},
when these are used to calculate the homogenized coefficients for this particular value of~$(X_1, X_2, X_3)$.
The static parts of all the estimators in~\cref{sec:cv_methodology} are calculated by numerical quadrature.

The time required on a personal laptop to evaluate all the estimators from~\cref{sec:cv_methodology},
not including the time required to train the neural networks,
is presented in \cref{table:runtime_multiscale} for~$K = 1000$ realizations.
To produce these numerical results,
an Euler--Maruyama discretization with a time step equal to $0.01 \varepsilon^2$ was employed to integrate the fast dynamics,
with the slow variables frozen at the value given above.
It appears from the tables that the estimators using adjoint control variates are faster to evaluate than
those using forward control variates,
and this difference is especially pronounced for the Green--Kubo estimators.
This is expected,
as the estimator~\eqref{eq:GK_cv_0_estimator} requires to evaluate the control variate only at the initial time. For half-Einstein, the additional cost of the forward estimator (compared to the adjoint estimator) is due to an additional gradient computation, present only in the forward case.

The tables also present the variance of all the estimators with~$K = 1$ realization,
when the integration time for the fast dynamics is set to $T/\varepsilon^2 = 5$.
These variances were calculated empirically based on 1000 independent realizations.
The third column, which contains the product of the runtime with the variance,
is a measure of the cost,
as in~\cref{subsec:lang_application}.
From this column we observe that,
both for the Green--Kubo and for the half-Einstein estimators,
the combined control variate method from~\cref{subsec:both_cv} performs best.
Furthermore the variance reduction obtained with the forward and adjoint control variates are similar.

\begin{table}[tbhp]
\footnotesize
    \captionsetup{position=top}
    \caption{Computational time required to calculate the homogenized drift and diffusion coefficients associated with the multiscale system~\eqref{eq:example_multiscale},
        based on $K = 1000$ realizations of the fast dynamics over a time horizon $T/\varepsilon^2 = 5$
        and for one fixed value of $(X_1, X_2, X_3) = (-0.0057, 1.73, -1.04)$. The cost is defined as the runtime times the variance.}\label{table:runtime_multiscale}
\begin{center}
    \subfloat[Green--Kubo data, normalized in terms of the standard Green--Kubo estimator runtime of $0.179$ seconds, and associated variance $29.382$.]{\label{table:runtime_multiscale_gk}
        \begin{tabular}{cccc} \toprule
        \textbf{Estimator} & \textbf{Runtime} & \textbf{Variance at $T/\varepsilon^2 = 5$} & \textbf{Cost} \\ \midrule
        GK~\eqref{eq:GK_estimator} & 1 & 1 & 1 \\
        GK/forward~\eqref{eq:GK_cv_t_estimator} & 77.0 & 0.00701 & 0.54  \\
        GK/adjoint~\eqref{eq:GK_cv_0_estimator} & 1.39 & 0.0128 & 0.0177 \\
        GK/combined~\eqref{eq:GK_cv_both_estimator} & 90.9 & 0.000104 & 0.00943 \\ \bottomrule
        \end{tabular}}

    \subfloat[Half-Einstein data, normalized in terms of the standard half-Einstein estimator runtime of $0.959$ seconds, and associated variance $3.44555$.]{\label{table:runtime_multiscale_he}
        \begin{tabular}{cccc} \toprule
        \textbf{Estimator} & \textbf{Runtime} & \textbf{Variance at $T/\varepsilon^2 = 5$} & \textbf{Cost} \\ \midrule
        HE~\eqref{eq:HE_gen_estimator} & 1 & 1 & 1 \\
        HE/forward~\eqref{eq:HE_cv_t_estimator} & 16.13 & 0.0107 & 0.172 \\
        HE/adjoint~\eqref{eq:HE_cv_0_estimator} & 4.20 & 0.0126 & 0.0529 \\
        HE/combined~\eqref{eq:HE_cv_both_estimator} & 18.45 & 0.000192 & 0.00354 \\ \bottomrule
    \end{tabular}}
\end{center}
\end{table}

The variances associated with all the estimators,
as a function of time, are illustrated in~\cref{fig:standard_deviations_gk,fig:standard_deviations_he} for the Green--Kubo and half-Einstein estimators, respectively.
These figures confirm the trends already observed in~\cref{table:runtime_multiscale_gk,table:runtime_multiscale_he}.
In particular, whether for Green--Kubo or half-Einstein estimators,
combining both control variates as in~\eqref{eq:GK_cv_both_estimator} and \eqref{eq:HE_cv_both_estimator} yields the largest variance reduction.
The figures also illustrate that the variance reduction obtained with the forward or adjoint control variates are similar.
\begin{figure}[ht]
    \centering
    \includegraphics[width=0.49\linewidth]{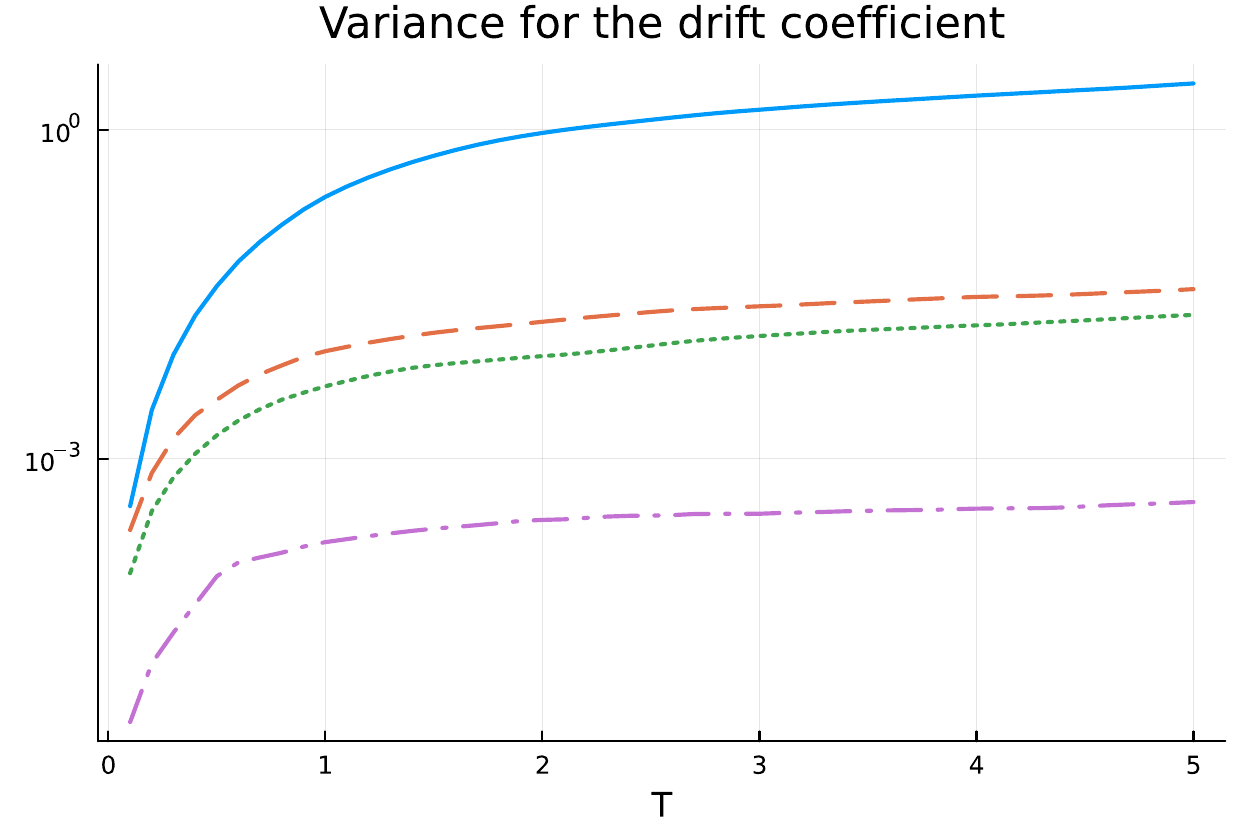}
    \includegraphics[width=0.49\linewidth]{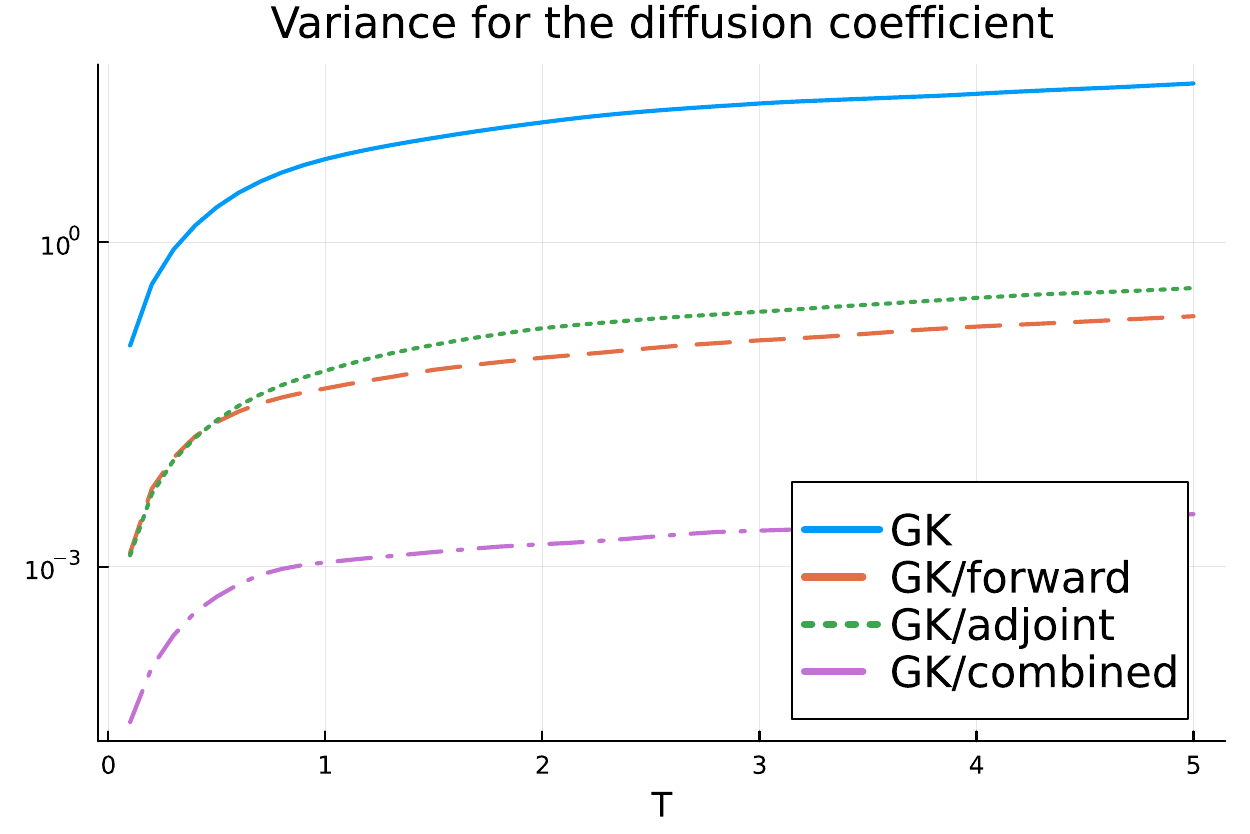}
    \caption{Variances of the \textbf{Green--Kubo} estimators presented in~\cref{sec:cv_methodology},
        when these are used to calculate the coefficients of the homogenized equation associated to~\eqref{eq:multiscale}.
        These values were calculated based on $K = 1000$ realizations of each estimator.
    }\label{fig:standard_deviations_gk}
\end{figure}

\begin{figure}[ht]
    \centering
    \includegraphics[width=0.49\linewidth]{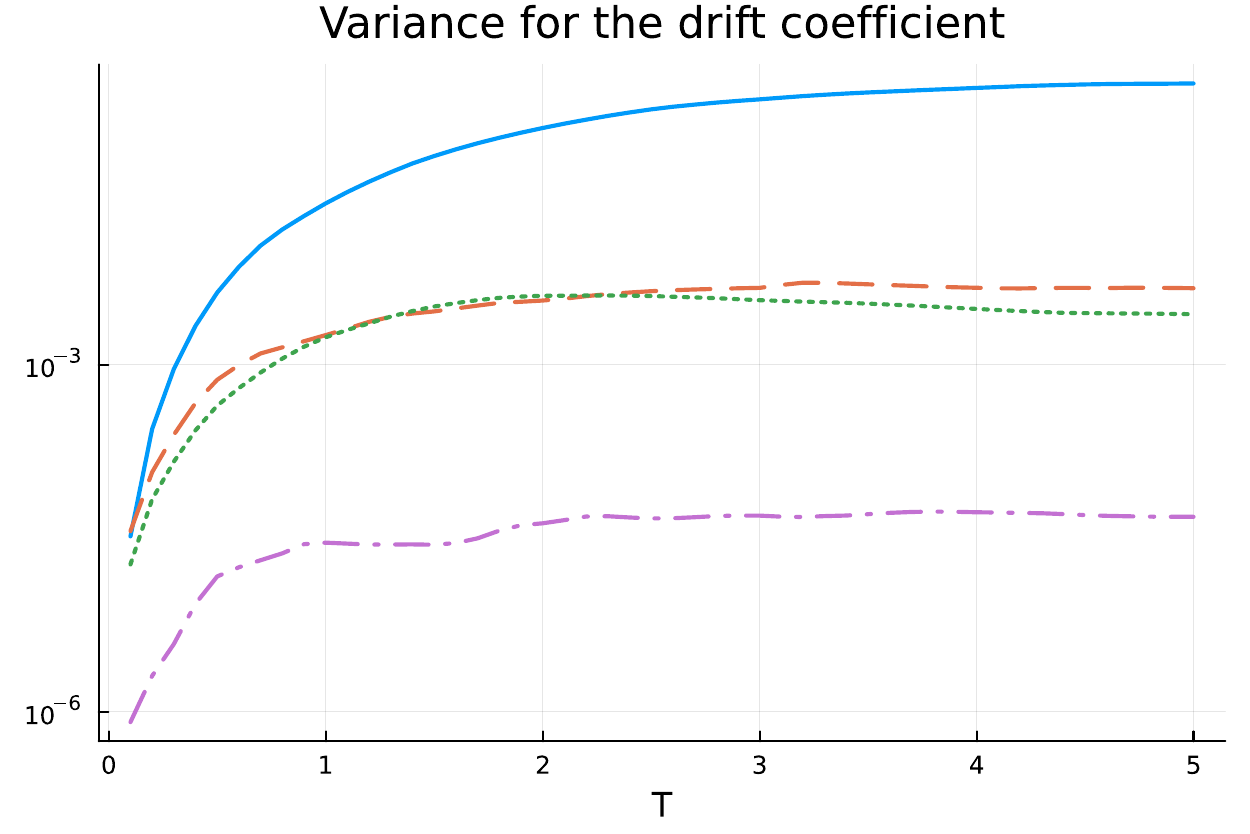}
    \includegraphics[width=0.49\linewidth]{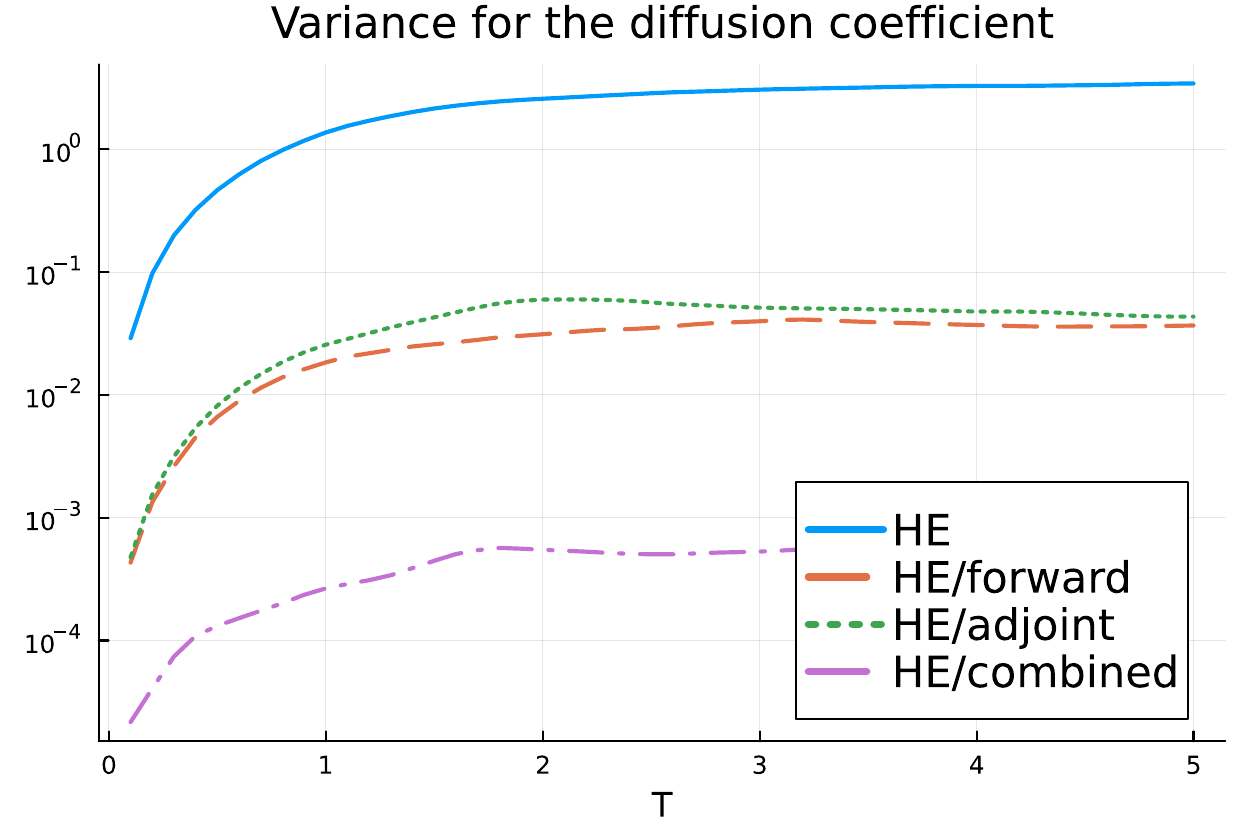}
    \caption{Variances of the \textbf{half-Einstein} estimators presented in~\cref{sec:cv_methodology},
        when these are used to calculate the coefficients of the homogenized equation associated to~\eqref{eq:multiscale}.
        These values were calculated based on $K = 1000$ realizations of each estimator.
    }\label{fig:standard_deviations_he}
\end{figure}

\section{Extensions and perspectives}
\label{sec:conclusion_pinns}
In this paper we demonstrated that neural networks can be an effective tool for constructing control variates via approximating Poisson equations used in the calculation of transport coefficients. The main point of our approach was to construct simple, easy-to-train and inexpensive-to-evaluate networks to be used as control variates, evaluated alongside Monte Carlo simulations. The key observation is that the network must be simple enough so that it is reasonably inexpensive to train and evaluate, while still providing significant variance reduction in order justify the additional computational cost. Clearly, in order for the proposed methodology to be of practical use, the cost (i.e., product of runtime and variance) must be less than that of the standard estimator.

There are several directions for extending the work. While possible, high-order nested automatic differentiation remains a nontrivial challenge, and is generally not an off-the-shelf tool, particularly for high-dimensional systems requiring batched training. Another possible direction is to further optimize the network's cost-to-variance reduction ratio, in particular by further exploring featurization approaches, which can potentially increase the network's efficiency without significantly increasing cost. Lastly, another possible extension of this method is for computing the mobility in the underdamped limit $\gamma\to 0$, particularly in the case of nonseparable potentials, which remains a challenge in the literature despite recent efforts \cite{pavliotis2023}.

\appendix
\section{Extension to general weights}
\label{appendix:generalizations_HE_var}
We state here some technical results which generalize the results of \cref{prop:var_standard_HE}. We first state an assumption which will be used in the results to come.

\begin{assumption}[Decay estimates of the semigroup in $L^4(\mu)$]
\label{as:semigroup_decay_L4}
There exist positive constants~$L_4$ and~$\lambda_4$ such that, for all $\varphi\in L^4_0(\mu)$ and all $t\geq 0$,
\begin{equation}
    \notag
    \norm*{\e^{t\L}\varphi}_{L^4_0(\mu)} \leq L_4\e^{-\lambda_4 t}\norm*{\varphi}_{L^4_0(\mu)}.
\end{equation}
\end{assumption}

\cref{as:semigroup_decay_L4} extends the semigroup decay estimates of \cref{as:semigroup_decay} from $L^2(\mu)$ to $L^4(\mu)$, a necessary estimate in \cref{lemma:regularity_asym_var} below.
By \cite[Theorem 2.2]{kusuoka2014},
see also~\cite[Theorem~1.3]{cattiaux2010}
and the discussion in \cite[Section~2.6]{pavliotis2014},
the exponential convergence estimate in~$L^2_0(\mu)$ given in~\cref{as:semigroup_decay}
implies similar exponential convergence estimates in $L^p_0(\mu)$,
for all $p \in (1, \infty)$ by an interpolation argument.
The same assumption is made on $\L^*$, and can be shown to hold using the same interpolation argument as above.

\begin{lemma}[Continuity estimate]
\label{lemma:regularity_asym_var}
Suppose that \cref{as:semigroup_decay_L4} holds true, and assume that $w$ is continuous at $0$.
Assume additionally that $f, g \in L^4_0(\mu)$.
Then, it holds that\begin{equation}
    \label{eq:statement}
\limsup_{T\to+\infty} \E\bkt*{\abs*{\HEest}^2} \leq \frac{L^2_4\norm{f}^2_{L^4(\mu)}\norm{g}^2_{L^4(\mu)}}{K}\paren*{\frac{\abs{w(0)}^2}{\lambda^2_4} + \frac{2}{\lambda_4}\int_{0}^1 w(u)^2 (1 - u) \, du}.
\end{equation}
\end{lemma}

\begin{remark}
    In contrast with the result of~\cref{prop:var_standard_HE}, in~\cref{lemma:regularity_asym_var} we make no regularity assumption on the derivatives of the weight function~$w$.
\end{remark}

\begin{proof}
As in the proof of \cref{prop:var_standard_HE},
we consider  $K=1$ and use the simplified notation $\HEvar$ for the estimator.

If the right-hand side of~\eqref{eq:statement} is infinite, then the inequality is satisfied,
so we assume from now on that this term is finite. We first rewrite, for $\HEvar$ defined in \eqref{eq:estimator},
\begin{equation}
    \label{eq:rewrite_var_estimator}
    \E\bkt*{\abs*{\HEvar}^2}
    = T^2 \int_{0}^{1} \!\! \int_{0}^{1} \!\! \int_{0}^{\tau} \!\! \int_{0}^{t} w(t-s)w(\tau - \varsigma)
    \E \Bigl[ f(X_{sT}) f(X_{\varsigma T}) g(X_{t T}) g(X_{\tau T}) \Bigr] \, ds \, d\varsigma \, dt \, d\tau.
\end{equation}
Fix $0 \leq a \leq b \leq c \leq d \leq T$,
as well as bounded functions~$h_1, h_2, h_3, h_4 \in L^4_0(\mu)$.
Conditioning successively on information at times $c$, $b$ and then $a$,
we obtain that
\begin{align*}
    \expect \bigl[ h_1(X_a) h_2(X_b) h_3(X_c) h_4(X_d) \bigr]
    &= \expect \left[ h_1(X_a) h_2(X_b) \left( h_3 \e^{(d-c) \mathcal L}h_4 \right) (X_c) \right] \\
    &= \expect \left[ h_1(X_a) \left( h_2 \e^{(c-b)\mathcal L}  \left( h_3 \e^{(d-c) \mathcal L}h_4 \right)  \right) (X_b) \right] \\
    &= \expect \left[ h_1(X_a) \e^{(b-a) \mathcal L} \left( h_2 \e^{(c-b)\mathcal L}  \left( h_3 \e^{(d-c) \mathcal L}h_4 \right)  \right) (X_a) \right] \\
    &= \left \langle h_1, \e^{(b-a) \mathcal L} \left( h_2 \e^{(c-b)\mathcal L}  \left( h_3 \e^{(d-c) \mathcal L}h_4 \right)  \right) \right \rangle\\
&= \left \langle h_2\e^{(b-a)\mathcal L^*} h_1, \e^{(c-b)\mathcal L}  \left( h_3 \e^{(d-c) \mathcal L}h_4 \right)   \right \rangle \\
&= \expect \left[(h_2\e^{(b-a) \mathcal L^{*}} h_1)(X_b)   \left( h_3 \e^{(d-c) \mathcal L}h_4 \right) (X_c) \right].
\end{align*}
It follows by~\cref{as:semigroup_decay_L4} that
\begin{align}
    \Bigl\lvert \E\bigl[ h_1(X_a) h_2(X_b) h_3(X_c) h_4(X_d) \bigr] \Bigr\rvert
    \notag
    &\leq \norm*{h_2\e^{(b-a)\L^*} h_1}\norm*{h_3\e^{(d-c)\L}h_4} \\
    \notag
    &\leq \norm*{h_2}_{L^4(\mu)}\norm*{h_3}_{L^4(\mu)}\norm*{\e^{(b-a)\L^*} h_1}_{L^4(\mu)}\norm*{\e^{(d-c)\L}h_4}_{L^4(\mu)} \\
    \label{eq:estimate_decay}
    &\leq L^2_4\norm{h_1}_{L^4(\mu)}\norm{h_2}_{L^4(\mu)}\norm{h_3}_{L^4(\mu)}\norm{h_4}_{L^4(\mu)}\e^{-\lambda_4(b-a)}\e^{-\lambda_4(d-c)}.
\end{align}
We shall use this estimate with times $t\geq s$ and $\tau\geq \varsigma$.
There are six possible orderings of these times,
which we group in two cases:
\[
    \text{ Case 1 }:
    \begin{cases}
        s \leq \varsigma \leq \tau \leq t, \\
        s \leq \varsigma \leq t \leq \tau, \\
        \varsigma \leq s \leq \tau \leq t, \\
        \varsigma \leq s \leq t \leq \tau.
    \end{cases}
    \qquad
    \text{ Case 2 }:
    \begin{cases}
        s \leq t \leq \varsigma \leq \tau, \\
        \varsigma \leq \tau \leq s \leq t.
    \end{cases}
\]
Using~\eqref{eq:estimate_decay},
we obtain the following estimate,
where the arguments of the maximum correspond to the two cases above:
\begin{align}
    \notag
    &\E\Bigl[ f(X_{sT}) f(X_{\varsigma T}) g(X_{tT}) g(X_{\tau T}) \Bigr] \\
    \label{eq:decay_estimate_2}
    &\qquad \leq L^2_4\norm{f}^2_{L^4(\mu)}\norm{g}^2_{L^4(\mu)} \max \Bigl\{ \e^{-\lambda_4 T|s- \varsigma| - \lambda_4 T |t-\tau|} , \e^{-\lambda_4 T|t-s| - \lambda_4 T |\tau-\varsigma|} \Bigr\},
\end{align}
which by substitution in~\eqref{eq:rewrite_var_estimator} leads to the inequality
\begin{equation}
    \notag
    \E\bkt*{\abs*{\HEvar}^2}
    \leq L^2_4\norm{f}^2_{L^4(\mu)}\norm{g}^2_{L^4(\mu)} \Bigl( A_1(T) + A_2(T) \Bigr),
\end{equation}
with $A_1$ and $A_2$ given by
\begin{align*}
    A_1(T) &= T^2\int_{0}^{1} \!\! \int_{0}^{1} \!\! \int_{0}^{\tau} \!\! \int_{0}^{t}
            w(t-s)w(\tau - \varsigma)
             \e^{-\lambda_4 T|s- \varsigma| - \lambda_4 T |t-\tau|}\, \d s \, \d \varsigma \, dt \, d\tau, \\
    A_2(T) &= T^2 \int_{0}^{1} \!\! \int_{0}^{1} \!\! \int_{0}^{\tau} \!\! \int_{0}^{t}
            w(t-s)w(\tau - \varsigma)
            \e^{-\lambda_4 T|t-s| - \lambda_4 T |\tau-\varsigma|}
            \, ds \, d\varsigma \, dt \, d\tau.
\end{align*}
The $A_1(T)$ term can be written as
\begin{equation*}
    A_1(T) = T \int_{0}^{1} \!\! \int_{0}^{1}
    \e^{- \lambda T |t-\tau|}
    f_T(t, \tau) \, dt \, d\tau,
    \qquad f_T(t, \tau) :=
    T \int_{0}^{\tau} \!\! \int_{0}^{t}
    w(t-s)w(\tau - \varsigma) \e^{-\lambda T|s- \varsigma|}
    \, ds \, d\varsigma.
\end{equation*}
By the Cauchy--Schwarz inequality, the term $f_T(t, \tau)$ can be bounded as follows:
\begin{align*}
    \bigl\lvert f_T(t, \tau) \bigr\rvert^2 &\leq \paren*{T \int_{0}^{\tau} \!\! \int_{0}^{t}
    \bigl\lvert w(t-s) \bigr\rvert^2 \e^{-\lambda T|s- \varsigma|}
    \, \d s \, \d \varsigma} \paren*{T \int_{0}^{\tau} \!\! \int_{0}^{t}
    \bigl\lvert w(\tau - \varsigma) \bigr\rvert^2 \e^{-\lambda T|s- \varsigma|}
    \, \d s \, \d \varsigma} \\
    &\leq
    \paren*{\frac{2}{\lambda} \int_{0}^{t}
    \bigl\lvert w(t-s) \bigr\rvert^2 \, \d s}\paren*{\frac{2}{\lambda} \int_{0}^{\tau}
    \bigl\lvert w(\tau - \varsigma) \bigr\rvert^2 \, \d \varsigma} \\
    &=
    \paren*{\frac{2}{\lambda}}^2
    \lVert w \rVert_{L^2([0, t])}^2 \lVert w \rVert_{L^2([0, \tau])}^2,
\end{align*}
where we used that
\begin{align*}
    T\int_0^\tau \e^{-\lambda T|s- \varsigma|} \, d\varsigma &= T\paren*{\int_0^s \e^{-\lambda T(s- \varsigma)} \, d\varsigma + \int_s^\tau \e^{-\lambda T(\varsigma - s)} \, d\varsigma} \\
    &= \frac{1 - \e^{-\lambda Ts}}{\lambda} + \frac{1 - \e^{-\lambda T(\tau - s)}}{\lambda} \leq \frac{2}{\lambda}.
\end{align*}
Thus, we obtain
\begin{align*}
    A_1(T)
    &\leq \frac{2}{\lambda} \int_{0}^1 \int_{0}^1 \lVert w \rVert_{L^2([0, t])} \lVert w \rVert_{L^2([0, \tau])} \, \d t \, \d \tau \\
    &= \frac{2}{\lambda} \left( \int_{0}^1 \lVert w \rVert_{L^2([0, t])} \, \d t \right)^2
    \leq \frac{2}{\lambda} \int_{0}^1 \lVert w \rVert_{L^2([0, t])}^2 \, \d t
    = \frac{2}{\lambda} \int_{0}^1 w(u)^2 (1 - u) \, \d u.
\end{align*}
For the second term $A_2(T)$, it is simple to show that (see \cref{rem:consistency_w})
\begin{equation}
    \notag
    A_2(T) = T^2\paren*{\int_{0}^{1} \int_0^t w(t-s) \, \e^{-\lambda T|t-s|} \, ds \, dt}^2 \xrightarrow[T\to\infty]{} \frac{w(0)^2}{\lambda^2},
\end{equation}
which concludes the proof.
\end{proof}

Using~\cref{lemma:regularity_asym_var},
we can use a density argument to generalize \cref{prop:var_standard_HE} to any weight function satisfying the assumptions of this lemma, with the additional requirement that the right-hand side of~\eqref{eq:statement} is finite.
This is the content of the following result.

\begin{corollary}
    \label{corollary:generalization}
    Suppose that \cref{as:semigroup_decay_L4} holds true.
    Assume that $w\colon [0, 1) \to \real$ is continuous at $0$ and that
    \begin{equation}
        \label{eq:assumption_corollary}
        \zeta_w = \int_{0}^1 w(u)^2 (1 - u) \, \d u < +\infty.
    \end{equation}
    Consider~$f,g \in L^4_0(\mu)$ such that $F=-\L^{-1}f$ and $G=-\L^{-1}g$ belong to $L^4(\mu)$,
    and assume in addition that $\chi_F,\chi_G \in L^4(\mu)$.
    Then,
    \begin{equation}
        \label{eq:statement_corollary}
        \lim_{T\to+\infty} \var\paren*{\HEest} = \frac{4\zeta_w}{K} \ip{f, -\L^{-1}f}\ip{g, -\L^{-1}g}.
    \end{equation}
\end{corollary}

\begin{proof}
As in the proof of \cref{lemma:regularity_asym_var}, we consider the case $K=1$ and the simplified notation $\HEvar$.
Let
\[
    \Omega(u) = \bigl(w(u) - w(0)\bigr) \sqrt{1-u}
\]
and fix $\varepsilon > 0$. By assumption, it holds that $\Omega \in L^2(0, 1)$,
so by density in $L^2[0, 1]$ of~$C_{\rm c}^{\infty}(0, 1)$,
there exists a function~$\Omega_{\varepsilon} \in C_{\rm c}^{\infty}(0, 1)$ such that
\begin{equation*}
    \lVert \Omega_{\varepsilon} - \Omega \rVert_{L^2(0, 1)} \leq \sqrt{\varepsilon}.
\end{equation*}
Let $w_{\varepsilon}(u) = w(0) + \Omega_{\varepsilon}(u) / \sqrt{1-u}$. By construction, it holds that $w_{\varepsilon}(0) = w(0)$ and
\begin{equation}
    \label{eq:density_close_weight}
    \int_0^1 \bigl\lvert w(u) - w_{\varepsilon}(u) \bigr\rvert^2 (1 - u) \, du
    = \int_0^1 \bigl\lvert \Omega(u) - \Omega_{\varepsilon}(u) \bigr\rvert^2 \, du
    \leq \varepsilon.
\end{equation}
We decompose the estimator as follows:
\begin{align*}
    \HEvar &= \frac{1}{T} \int_{0}^{T} w_\varepsilon \paren*{\frac{u}{T}}\int_{0}^{T-u} f(X_s) g(X_{s+u}) \, ds \, du \\
           &\qquad + \frac{1}{T} \int_{0}^{T} (w - w_{\varepsilon}) \paren*{\frac{u}{T}}\int_{0}^{T-u} f(X_s) g(X_{s+u}) \, ds \, du
           := \widehat r_{T,\varepsilon} + \widehat \varrho_{T,\varepsilon}.
\end{align*}
Taking the variance and rearranging,
we deduce that
\begin{equation*}
    \var\paren*{\HEvar} - \var\paren*{\widehat r_{T, \varepsilon}}
    = \var\paren*{\widehat \varrho_{T,\varepsilon}} + 2 \cov \paren*{\widehat r_{T, \varepsilon}, \widehat \varrho_{T,\varepsilon}}.
\end{equation*}
Taking the absolute value, the limit superior as $T \to +\infty$,
and using Young's inequality, we obtain \begin{align}
    \notag
    \limsup_{T\to\infty} \Bigl\lvert \var\paren*{\HEvar} - \var\paren*{\widehat r_{T, \varepsilon}} \Bigr\rvert
    &= \limsup_{T\to\infty} \Bigl\lvert \var(\widehat \varrho_{T,\varepsilon}) + 2 \cov\bkt*{\widehat r_{T, \varepsilon}, \widehat\varrho_{T,\varepsilon}} \Bigr\rvert \\
    \label{eq:last_equation_complement}
    &\leq \limsup_{T\to\infty} \paren*{\var(\widehat \varrho_{T,\varepsilon}) + \sqrt{\varepsilon} \var(\widehat r_{T, \varepsilon}) + \frac{1}{\sqrt{\varepsilon}}\var(\widehat \varrho_{T,\varepsilon})}.
\end{align}
By~\cref{lemma:regularity_asym_var} and~\eqref{eq:density_close_weight},
and by~\cref{prop:var_standard_HE},
it holds that
\begin{align*}
    \limsup_{T \to +\infty} \var(\widehat \varrho_{T,\varepsilon})
    &\leq C_{f,g} \varepsilon, \\
    \lim_{T \to +\infty} \var(\widehat r_{T,\varepsilon})
    &= 4\zeta_{w_\varepsilon} \ip{f, -\L^{-1} f}\ip{g, -\L^{-1} g}.
\end{align*}
Using these equations in~\eqref{eq:last_equation_complement}
and taking the limit $\varepsilon \to 0$,
we deduce that
\[
    \lim_{\varepsilon\to0} \limsup_{T\to\infty} \abs*{\var\paren*{\HEvar} - 4\zeta_{w_\varepsilon} \ip{f, -\L^{-1} f}\ip{g, -\L^{-1} g}}
    = 0,
\]
which gives the desired estimate~\eqref{eq:statement_corollary}.
\end{proof}

It would be satisfying to show that when assumption~\eqref{eq:assumption_corollary} is not satisfied,
then the asymptotic value of the variance is infinite,
which would enable to further generalize \cref{corollary:generalization}.
Proving this in general is not simple, but we show the following result.

\begin{corollary}
    \label{corollary:unbounded_variance}
    Suppose that $f,g \in L^4_0(\mu)$ and that $w \colon [0,1) \to [0,+\infty)$ is continuous at 0
    and satisfies the condition
    \begin{equation}
        \label{eq:assumption}
        \mathscr{W} := \esssup_{\theta \in (0, 1]} \theta w(1 - \theta) < \infty.
    \end{equation}
    Then, it holds that
    \begin{equation}
        \label{eq:statement_corollary_general}
        \lim_{T\to\infty} \var\paren*{\HEest} = \frac{4\zeta_w}{K} \ip{f, -\L^{-1}f}\ip{g, -\L^{-1}g}.
    \end{equation}
    (Here the right-hand side is allowed to be infinite.)
\end{corollary}
\begin{remark}
    The condition~\eqref{eq:assumption} expresses that $w$ does not diverge faster than $(1-\theta)^{-1}$ at~$\theta = 1$.
    The weight function $w(\theta) = (1 - \theta)^{-1}$,
    for example, satisfies this condition.
\end{remark}

\begin{proof}
    For conciseness,
    let $\widehat \rho_T = \HEest$.
    If the right-hand side of \eqref{eq:statement_corollary_general} is finite, then the inequality follows from \cref{corollary:generalization}, so we assume from now on that
\begin{equation*}
        \int_{0}^1 w(u)^2 (1 - u) \, \d u = \infty.
    \end{equation*}
For $\varepsilon \in (0, 1]$
    consider the following decomposition of the weight:
\begin{equation*}
        w = \mathbf 1_{[0, 1 - \varepsilon]} w + \mathbf 1_{[1 - \varepsilon, 1]} w
        =: w_{\varepsilon} + \omega_{\varepsilon}.
    \end{equation*}
We decompose the estimator as follows:
\begin{align*}
        \widehat \rho_T
    &= \frac{1}{T} \int_{0}^{T} w_\varepsilon \paren*{\frac{u}{T}}\int_{0}^{T-u} f(X_s) g(X_{s+u}) \, \d s \, \d u \\
    &\qquad + \frac{1}{T} \int_{0}^{T} \omega_\varepsilon \paren*{\frac{u}{T}} \int_{0}^{T-u} f(X_s) g(X_{s+u}) \, \d s \, \d u
    =: \widehat r_{T,\varepsilon} + \widehat \varrho_{T,\varepsilon}.
    \end{align*}
The proof of~\eqref{eq:statement_corollary_general} is based on the following simple inequality:
\begin{equation}
        \label{eq:covariance}
        \expect\bigl[ \widehat \rho_{T} ^2\bigr]
        = \expect\bigl[ \widehat r_{T,\varepsilon} ^2\bigr]
        + \expect\bigl[ \widehat \varrho_{T,\varepsilon} ^2\bigr]
        + 2 \expect\bigl[ \widehat r_{T,\varepsilon} \widehat \varrho_{T,\varepsilon} \bigr]
        \geq
        \expect\bigl[ \widehat r_{T,\varepsilon} ^2\bigr]
        + 2 \expect\bigl[ \widehat r_{T,\varepsilon} \widehat \varrho_{T,\varepsilon} \bigr],
    \end{equation}
so that
\begin{equation}
        \Var(\HEvar) \geq \expect\bigl[ \widehat r_{T,\varepsilon} ^2\bigr]
        + 2 \expect\bigl[ \widehat r_{T,\varepsilon} \widehat \varrho_{T,\varepsilon} \bigr] - \E[\HEvar]^2.
        \label{eq:var_ineq_decomp}
    \end{equation}
It follows from~\eqref{eq:statement} that
\begin{equation*}
        \lim_{T\to+\infty} \E\bigl[ \widehat r_{T,\varepsilon} ^2\bigr]
        =
        4\ip{f, -\L^{-1} f}
        \ip{g, -\L^{-1} g}
        \int_{0}^{1} \bigl\lvert w_\varepsilon(\theta) \bigr\rvert^2 (1 - \theta) \, \d \theta.
    \end{equation*}
This term diverges as $\varepsilon \to 0$.
    It remains to analyze the last term on the right-hand side of~\eqref{eq:covariance}.
    Using Fubini's theorem and a change of variable, we have
\begin{align*}
        \E\bigl[ \widehat r_{T,\varepsilon} \widehat \varrho_{T,\varepsilon} \bigr]
        &=
        T^2 \int_{0}^{1} \!\! \int_{0}^{1}
        w_\varepsilon \left(1 - u \right) \omega_\varepsilon \left(1 - v\right)
        \int_{0}^{u} \!\!  \int_{0}^{\upsilon} E(s, \varsigma, u, \upsilon) \, \d s \, \d \varsigma \, \d u \, \d \upsilon,
    \end{align*}
where $E(s, \varsigma, u, \upsilon) = \expect \bigl[ f(X_{sT}) f(X_{\varsigma T}) g(X_{sT+(1-u)T}) g(X_{\varsigma T+(1-\upsilon)T}) \bigr]$.
    It follows from~\eqref{eq:decay_estimate_2}, with~$t = s + 1 - u$ and $\tau = \varsigma + 1 - \upsilon$, that
\begin{equation*}
        E(s, t, u, v)
        \leq  C \max \Bigl\{ \e^{-\lambda T|s- \varsigma| - \lambda T |s - \varsigma + \upsilon - u|}, \e^{-\lambda T|1-u| - \lambda T |1-\upsilon|} \Bigr\}.
    \end{equation*}
By the triangle inequality,
    it holds that $|s - \varsigma| + |s - \varsigma + \upsilon - u| \geq |\upsilon - u|$
    and $|1-u| + |1-\upsilon| \geq |u-\upsilon|$,
    so we deduce that
\begin{equation*}
        \forall (s, \varsigma, u, \upsilon) \in [0, 1]^4, \qquad
        \forall (s, \varsigma, u, \upsilon) \in [0, 1]^4, \qquad
        E(s, \varsigma, u, \upsilon) \leq C \e^{-\lambda|u-v|T}.
\end{equation*}
Therefore, using this inequality,
the definition of $w_{\varepsilon}$ and $\omega_{\varepsilon}$, and~\eqref{eq:assumption},
we obtain
\begin{align*}
    \Bigl\lvert \expect\bigl[ \widehat r_{T,\varepsilon} \widehat \varrho_{T,\varepsilon} \bigr] \Bigr\rvert
        &\leq
        C T^2 \int_{0}^{1} \!\! \int_{0}^{1}
        w_\varepsilon \left(1 - u \right) \omega_\varepsilon \left(1 - v\right)
        \e^{- \lambda |u-v|T} u v \, \d u \, \d \upsilon \\
        &=
        C T^2 \int_{0}^{\varepsilon} \!\! \int_{\varepsilon}^{1}
        w \left(1 - u \right) w \left(1 - \upsilon\right)
        \e^{- \lambda |u-\upsilon|T} u \upsilon \, \d u \, \d \upsilon \\
        &\leq
        C T^2 \mathscr{W}^2 \int_{0}^{\varepsilon} \!\! \int_{\varepsilon}^{1}
        \e^{- \lambda |u-\upsilon|T} \d u \, \d \upsilon \leq \frac{C \mathscr{W}^2}{\lambda^2},
\end{align*}
where the last inequality follows from the fact that $|u-v| = u-v$ (since~$u\in[\varepsilon,1]$ and~$v\in[0,\varepsilon]$), so that
\begin{align*}
    T^2 \int_0^\varepsilon \int_\varepsilon^1
        \e^{-\lambda |u-\upsilon|T} \d u \, \d \upsilon &= \frac{(\e^{\lambda T \varepsilon} - 1)(\e^{-\lambda T \varepsilon}-\e^{-\lambda T})}{\lambda^2} \leq \frac{1}{\lambda^2}.
\end{align*}
The statement~\eqref{eq:statement_corollary_general} then follows from taking the limits $T \to +\infty$
and then $\varepsilon \to 0$ in~\eqref{eq:var_ineq_decomp}.
\end{proof}

\section{Change of variables}
\label{appendix:cov}

\begin{lemma}
    \label{lemma:cov2}
    Let $f$ be a continuous function on $\mathcal{D} = \{(t,s) \colon 0\leq s\leq t, \; t \in [0,T]\}$. Then,
    \begin{equation}
        \int_0^T\int_0^t f(t-s) \, ds \, dt = \int_0^T f(\theta)(T - \theta) \, d\theta.
        \label{eq:cov_lemma_statement}
    \end{equation}
\end{lemma}

\begin{proof}
Define the variables $\theta = t-s$ and $\tau = t$, and consider the $C^1$-diffeomorphism
\begin{equation}
    \notag
    \Phi(\tau, \theta) = (t(\tau,\theta), s(\tau,\theta)) = (\tau, \tau - \theta).
\end{equation}
It holds that
\begin{align}
     \iint_\mathcal{D} f(t-s) \, ds \, dt &= \iint_{\widetilde{\mathcal{D}}} f\paren*{\Phi_1(\tau,\theta) - \Phi_2(\tau,\theta)} \abs*{J_\Phi(\tau,\theta)} \, d\theta \, d\tau,
     \label{eq:cov_cor}
\end{align}
with $\widetilde{\mathcal{D}} = \Phi^{-1}(\mathcal{D})$ given by
\begin{equation}
    \notag
    \widetilde{\mathcal{D}} = \{(\tau,\theta) \colon 0\leq \theta\leq\tau, \; \tau\in[0,T]\}.
\end{equation}
The Jacobian is given by
\begin{equation}
    \notag
    J_\Phi(\tau,\theta) = \det\begin{bmatrix}
        \dfrac{\partial t(\tau,\theta)}{\partial \tau} & \dfrac{\partial t(\tau,\theta)}{\partial\theta} \\[1em]
        \dfrac{\partial s(\tau,\theta)}{\partial\tau} & \dfrac{\partial s(\tau,\theta)}{\partial\theta}
    \end{bmatrix} = \det\begin{bmatrix}
        1 & 0 \\
        1 & -1
    \end{bmatrix} = -1,
\end{equation}
so that $|J_\Phi(\tau,\theta)| = 1$. Thus, applying \eqref{eq:cov_cor}, using Fubini's theorem then integrating with respect to~$\tau$ gives
\begin{align}
    \notag
    \int_0^T\int_0^t f(t-s) \, ds \, dt &= \int_0^T \int_0^\tau f(\theta) \, d\theta \, d\tau = \int_0^T \int_\theta^T f(\theta) \, d\tau \, d\theta = \int_{0}^{T} f(\theta)(T-\theta) \, d\theta,
\end{align}
which is the desired result.
\end{proof}

\begin{corollary}
\label{cor:cov3}
    Under the same setting as \cref{lemma:cov2}, it holds that
\begin{equation}
    \notag
    \int_0^T\int_0^t f\paren*{\frac{t-s}{T}} \, ds \, dt = T^2\int_0^1 f(v)(1 - v) \, dv.
\end{equation}
\end{corollary}

\begin{proof}
    This result follows trivially from \cref{lemma:cov2} with the additional change of variable~$\theta = Tv$ in~\eqref{eq:cov_lemma_statement}.
\end{proof}

\section{Additional material on fluctuation formulas}
\label{appendix:add_fluc_form}
\subsection{Weight functions numerical illustration}
\label{appendix:weight_num_ill}
In this section,
we illustrate the bias and variance of the estimators~\eqref{eq:GK_estimator}  and~\eqref{eq:HE_gen_estimator},
with various weights for the latter,
in a simple setting where these quantities can be calculated explicitly,
up to low-dimensional numerical quadratures.
Specifically, we consider the setting where $X_t$ is the one-dimensional stationary Ornstein--Uhlenbeck process,
\[
    \d X_t = - X_t \, \d t + \sqrt{2} \, \d W_t, \qquad X_0 \sim \mathcal N(0, 1).
\]
and $f,g$ are the identity functions.
That is to say,
we consider estimators for the quantity
\[
    \rho = \int_{0}^{\infty} k(t) \, \d t, \qquad k(t) := \expect \left[ X_0 X_t \right] = \e^{-\abs{t}}.
\]

\paragraph{Green--Kubo estimator}A simple calculation gives that the bias and variance of the estimator~\eqref{eq:GK_estimator},
for~$K = 1$ realization, are given by
\begin{align*}
    &\expect \bkt*{\GKest} - \rho = -\e^{-T}, \\
    &\var \bkt*{\GKest}
    = \int_{0}^{T} \int_{0}^{T} \expect {\left[ X^1_0 \, X^1_0 \, X^1_s \, X^1_t \right]} \, \d s \, \d t - \expect {\left[ \GKest \right]}^2 \\
    &\hspace{1.5cm} = \int_{0}^{T} \int_{0}^{T} 2 k(t) k(s) + k(0) k(t-s) \, \d s \, \d t - \expect {\left[ \GKest \right]}^2
    = 2T - 1 + \e^{-2T},
\end{align*}
where we used Isserlis' theorem for the variance.
The absolute bias and the variance are illustrated for a range of values of~$T$ in~\cref{fig:bias_var_gk}.
\begin{figure}[ht]
    \centering
    \includegraphics[width=0.9\linewidth]{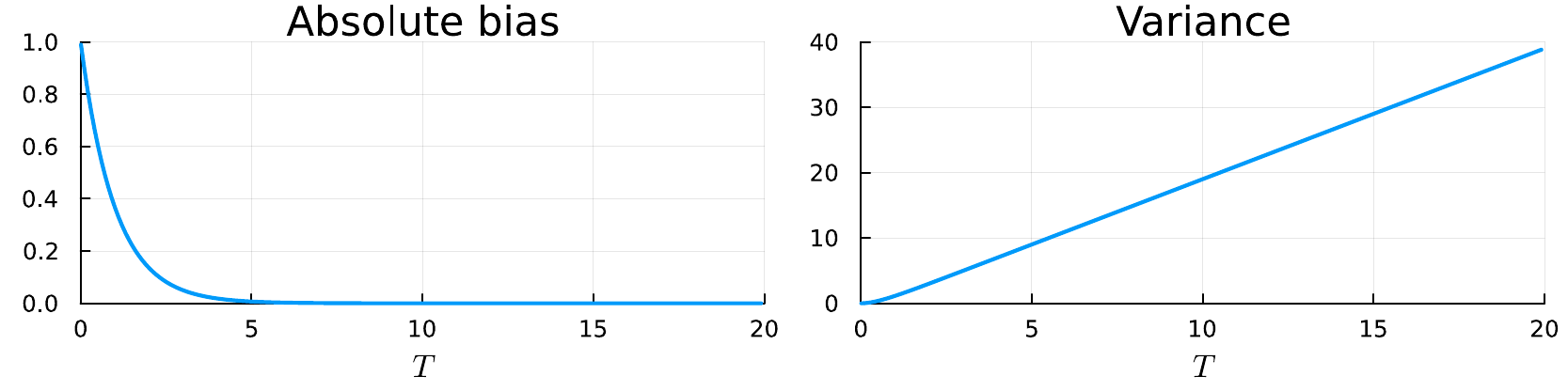}
    \caption{Absolute bias and variance of the \textbf{Green--Kubo} estimator~\eqref{eq:GK_estimator},
        in the simple setting of \cref{appendix:weight_num_ill}.
    }\label{fig:bias_var_gk}
\end{figure}

\paragraph{Half-Einstein estimators}\label{par:Half-Einstein}
A similar reasoning
gives that the bias of the half-Einstein estimator~\eqref{eq:HE_gen_estimator} is given by
\begin{align*}
    &\expect \bkt*{\HEest} - \rho
= \int_{0}^{T} \paren*{1 - \frac{\theta}{T}} w \Bigl(\frac{\theta}{T}\Bigr) k(\theta) \, \d \theta - 1.
\end{align*}
To obtain a practical formula for the variance,
we first rewrite the estimator~\eqref{eq:HE_gen_estimator} as follows:
\begin{equation}
    \label{eq:rewriting}
    \HEest = \frac{1}{TK} \sum_{k=1}^{K} \int_0^T\int_0^{T-u} w \paren*{\frac{u}{T}} f(X^k_s) g(X^k_{s+u}) \, ds \, du.
\end{equation}
Consider the particular setting of this subsection, with $f,g$ the identity functions and $K = 1$.
Squaring~\eqref{eq:rewriting}, taking the expectation, and using Isserlis' theorem,
we obtain the following expression for the variance:
\[
    \var \bkt*{\HEest} =  \frac{1}{T^2} \int_0^T \! \! \int_0^T w \paren*{\frac{u}{T}} w \paren*{\frac{v}{T}} I(u, v) \, \d v \, \d u - \expect \bkt*{\HEest}^2,
\]
where
\[
    I(u, v) := \int_0^{T-u} \int_0^{T-v} \Bigl( k(u) k(v) + k(t+v-s-u) k(t-s) + k(t+v-s) k(t-s-u) \Bigr) \, \d t \, \d s.
\]
The inner double integral $I(u, v)$ can be calculated explicitly by symbolic calculations,
after which the outer double integral can be calculated by numerical quadrature.
The absolute bias and the variance,
for various weight functions encountered in the literature and made precise in \cref{table:weights_num_illustration}, and for a range of values of~$T$,
are illustrated  in~\cref{fig:bias_var_he}.
\begin{figure}[ht]
    \centering
    \includegraphics[width=0.9\linewidth]{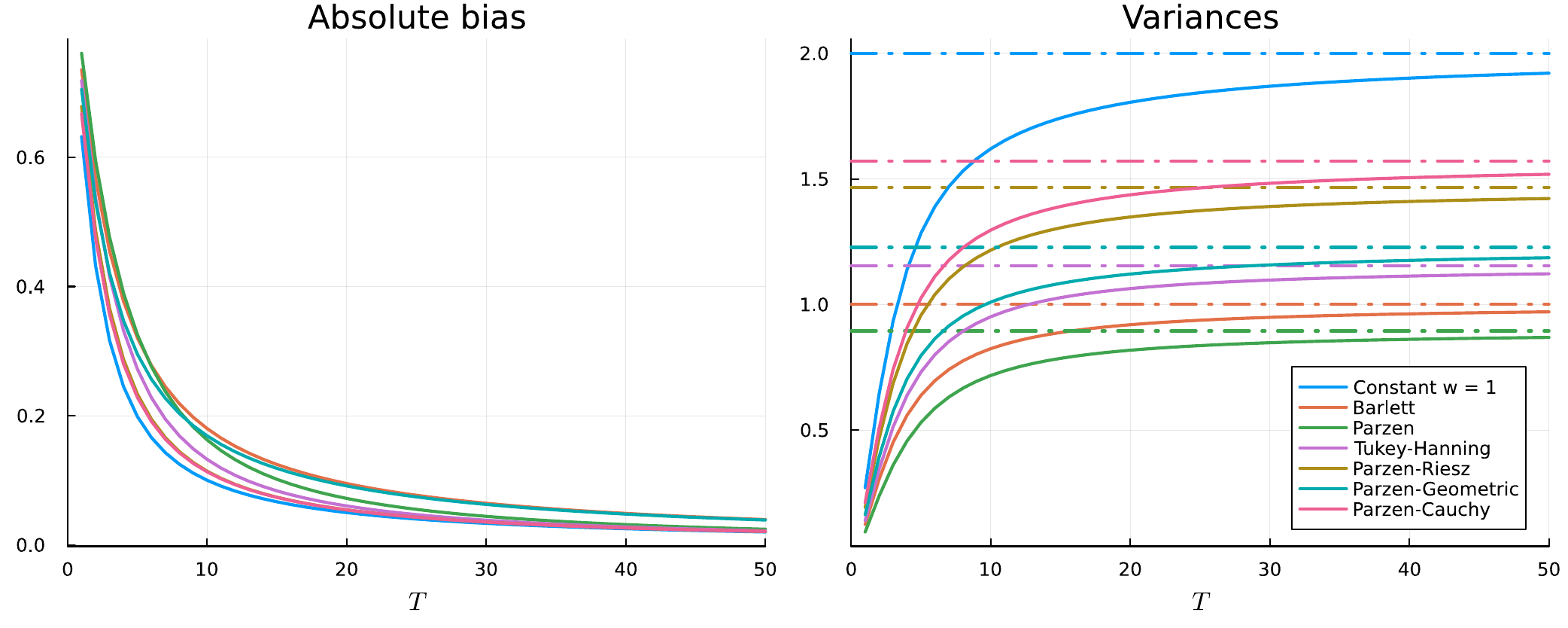}
    \caption{Absolute bias and variance of the \textbf{half-Einstein} estimator~\eqref{eq:HE_gen_estimator} for the weight functions described in \cref{table:weights_num_illustration}
        in the simple setting of \cref{appendix:weight_num_ill}.
        The dashed lines correspond to the asymptotic values of the variances given by~\cref{prop:var_standard_HE}.
    }\label{fig:bias_var_he}
\end{figure}

\begingroup
\renewcommand{\arraystretch}{1.2}
\begin{table}[htb!]
    \footnotesize
    \centering
    \caption{Expression for various weight functions $w(t)$ for $0\leq t\leq 1$, otherwise defined as~$0$; see \cite{wang2012} and references therein.}
    \label{table:weights_num_illustration}
    \begin{tabular}{ll}
        \toprule
        \textbf{Weight function} & \textbf{Expression for $w(t)$}\\\midrule
        Constant & $w(t) = 1$ \\
        Bartlett & $w(t) = 1 - t$ \\
        Parzen & $w(t) = \begin{cases} 1 - 6t^2 + 6t^3, &t\leq 0.5 \\ 2(1-t)^3, &t > 0.5\end{cases}$ \\
        Tukey--Hanning & $w(t) = \dfrac{1+\cos(\pi t)}{2}$ \\
        Parzen--Riesz & $w(t) = 1 - t^2$ \\
        Parzen--Geometric & $w(t) = \dfrac{1}{1+t}$ \\
        Parzen--Cauchy & $w(t) = \dfrac{1}{1+t^2}$ \\
        \bottomrule
    \end{tabular}
\end{table}
\endgroup

\subsection{Remark on the half-Einstein estimator}
\label{appendix:HE_remark}
\begin{remark}
    The estimator~\eqref{eq:HE_gen_estimator} may also be rewritten as
    \[
        \frac{1}{TK}\sum_{k=1}^K \int_0^T\int_0^{T-u} w\paren*{\frac{u}{T}} f(X_s^k)g(X_{s+u}^k) \, ds \, du.
    \]
    Thus, the estimator could also be interpreted as an approximation of the Green--Kubo formula~\eqref{eq:GK} where the correlation~$\expect \left[ f(X_0) g(X_u) \right]$
    is approximated by
    \[
        \expect \left[ f(X_0) g(X_u) \right]
        \approx
        \frac{1}{TK}\sum_{k=1}^K \int_0^{T-u} w\paren*{\frac{u}{T}} f(X_s^k)g(X_{s+u}^k) \, ds,
        \qquad 0 \leq u \leq T.
    \]
    The right-hand side is an unbiased estimator of the left-hand side for $w(x) = (1-x)^{-1}$.
    However, this choice of the weight function leads to the estimator~\eqref{eq:HEest_no_w} having an unbounded variance in the limit as~$T \to \infty$,
    as suggested by~\cref{prop:var_standard_HE} further in this section
    and proved rigorously in \cref{corollary:unbounded_variance}.
    In practice, other choices of $w$ are thus preferred.
\end{remark}

\section{Neural network architecture}
\label{appendix:nn_architecture}
As discussed in \cref{subsubsec:num_results_2d-lang} for the Langevin dynamics example, we employ a featurization layer to transform the original four inputs $(q_1, q_2, p_1, p_2)$ into a seven-dimensional feature set, ensuring the periodicity of positional terms and augmenting the feature space with additional relevant kinetic terms. This is done via \eqref{eq:featurization}, which we recall for convenience:
\begin{equation}
    \notag
    \mathrm{featurization}\paren*{q_1,q_2,p_1,p_2} = \paren*{\sin(q_1), \cos(q_1), \sin(q_2), \cos(q_2), p_1, p_2, \frac{p_1^2 + p_2^2}{2}}.
\end{equation}
Typically, such featurization approaches would directly use the transformed features as the input layer without an additional ``functional layer''. However, by retaining the original four inputs before featurization, we avoid the complex chain-rule transformations in the loss function, where the differential operator \eqref{eq:lang_generator} acts on the solution with respect to the original variables $(q_1, q_2, p_1, p_2)$.

The network topology is illustrated in \cref{fig:2d-lang_network_both} for the Langevin and multiscale dynamics examples of \cref{subsec:lang_application,subsec:multiscale}, respectively.

\begin{figure}[tbh]
    \centering
    \includegraphics[width=0.7\linewidth]{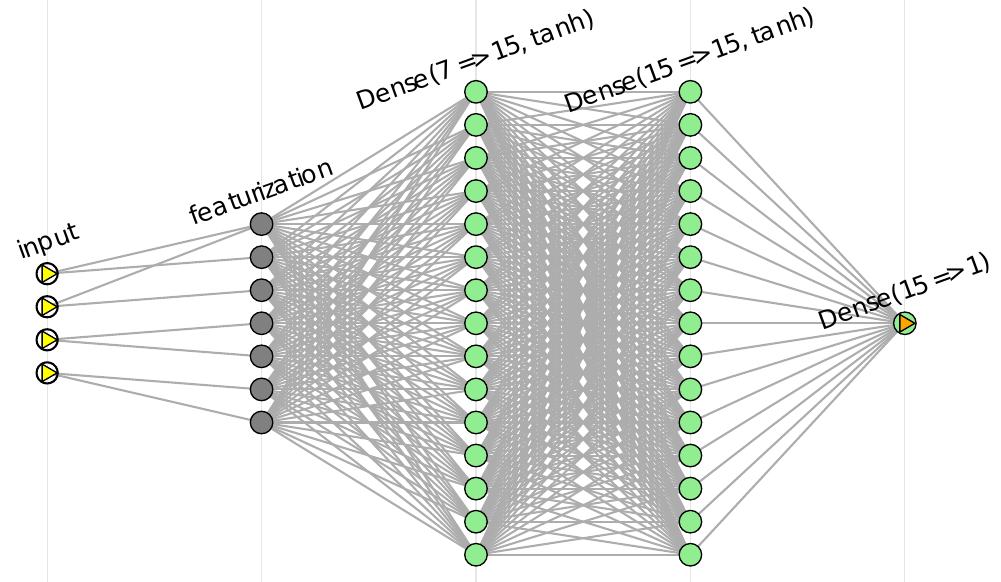}
    \includegraphics[width=0.7\textwidth]{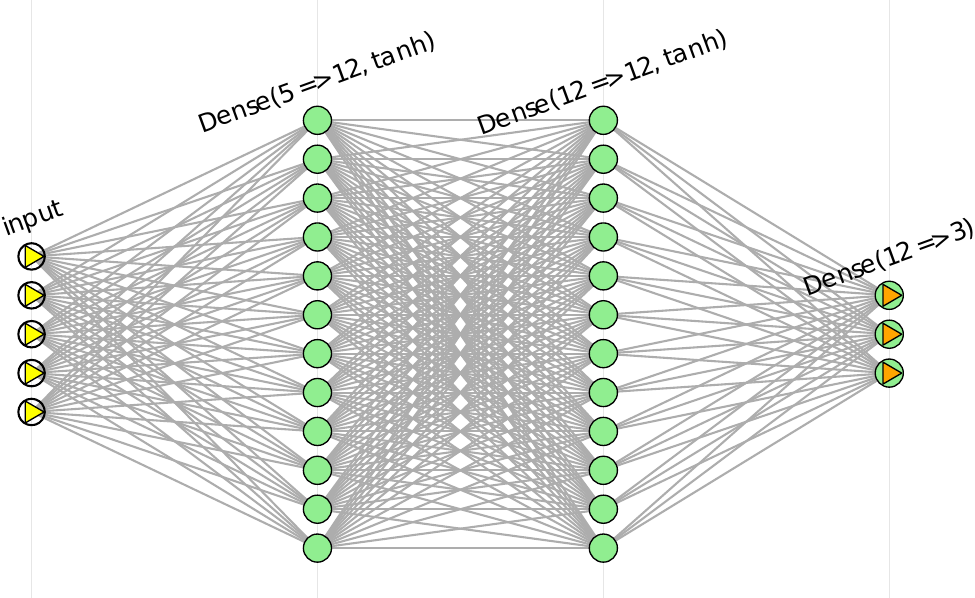}
    \caption{Illustration of neural network architecture. Top: Langevin dynamics, with inputs undergoing a featurization layer. Bottom: multiscale dynamics. Both examples employ $\tanh$ as the activation function for the hidden layers.} \label{fig:2d-lang_network_both}
\end{figure}

\section*{Acknowledgments}
The authors would like to thank Petr Plechac and George Kevrekidis for helpful discussions.
This project has received funding from the European Union's Horizon 2020 research and innovation program under the Marie Sklodowska--Curie grant agreement No 945332, and from the European Research Council (ERC) under the European Union's Horizon 2020 research and innovation programme (project EMC2, grant agreement No 810367). We also acknowledge funding from the Agence Nationale de la Recherche, under grants ANR-19-CE40-0010-01 (QuAMProcs), ANR-21-CE40-0006 (SINEQ) and ANR-23-CE40-0027 (IPSO).
GP is partially supported by an ERC-EPSRC Frontier Research Guarantee through Grant No. EP/X038645, ERC Advanced Grant No. 247031 and a Leverhulme Trust Senior Research Fellowship, SRF$\backslash$R1$\backslash$241055. Part of the work was done while RS was visiting the CNRS-Imperial Abraham de Moivre IRL. We thank the Lab for the hospitality and partial financial support.

\clearpage
\printbibliography

\end{document}